\documentclass[a4paper,11pt,draft]{article}

\usepackage{amsmath,amsfonts,amsthm,amssymb}

\newtheorem{Th}{Theorem}[section]
\newtheorem{Lem}[Th]{Lemma}
\newtheorem{Prop}[Th]{Proposition}
\newtheorem{Cor}[Th]{Corollary}
\newtheorem{Defi}[Th]{Definition}

\newtheorem{Exam}[Th]{Example}

\DeclareMathOperator*{\LIM}{l.i.m.}

\newcommand{\R}{\mathbb{R}}
\newcommand{\C}{\mathbb{C}}
\newcommand{\N}{\mathbb{N}}

\newcommand{\cB}{{\cal B}}

\newcommand{\I}{{\infty}}

\newcommand{\cl}{{\cal L}}
\newcommand{\cfl}{{\cal{FL}}}
\newcommand{\cL}{{\cal L}^2(\R^{\infty})}
\newcommand{\cLP}{{\cal L}_b^2(\R^{\infty})}
\newcommand{\cLF}{{\cal F}{\cal L}^2(\R^{\infty})}
\newcommand{\LL}{\mathbb{L}^2(\R^{\infty})}

\newcommand{\RI}{{\mathbb{R}}^{\infty}}
\newcommand{\RIO}{{\mathbb{R}}_0^{\infty}}

\newcommand{\la}{\left\langle}
\newcommand{\ra}{\right\rangle}

\newcommand{\ssqrt}[1]
{\sqrt{\smash[b]{\mathstrut #1}}}

\begin{document}
%%%%%%%%%%%%%%%%%%%%%%%%%%%%%%%%%%%%%%%%%%%%%%%%%%%%%%%%%%%%%%%%%%%%%%%%%%%%
\title{{\Large{\bf L\'{e}vy Laplacian for Square Roots of Measures}}}
\author{Hiroaki Kakuma \\
Department of Mathematics,\\Nagoya University, Nagoya 466-8850, Japan,\\
\texttt{d08001p@nagoya-u.ac.jp} }

%%%%%%%%%%%%%%%%%%%%%%%%%%%%%%%%%%%%%%%%%%%%%%%%%%%%%%%%%%%%%%%%%%%%%%%%%%%%
\date{}
\renewcommand{\baselinestretch}{1.2}
%%%%%%%%%%%%%%%%%%%%%%%%%%%%%%%%%%%%%%%%%%%%%%%%%%%%%%%%%%%%%%%%%%%%%%%%%%%%%

\maketitle

\begin{abstract} 
L. Accardi showed that a Banach space of signed measures is homeomorphic to a Hilbert space 
formed by the so-called square roots of measures.
In this paper, we redefine square roots of measures in view of the theory of measures on infinite dimensional spaces, 
and introduce notions of differentiation, Fourier transform, and convolution product of square roots of measures,
and examine those relations. By using these tools,
we study the L\'{e}vy Laplacian for squares root of measures including non-Gaussian type. 
It is shown that the symbol of the L\'{e}vy Laplacian is equal to the quadratic variation of paths. 
\end{abstract}

%%%%%%%%%%%%%%%%%%%%%%%%%%%%%%%%%%%%%%%%%%%%%%%%%%%%%%%%%%%%%%%%%%%%%%%%%%%%
%%%%%%%%%%%%%%%%%%%%%%%%%%%%%%%%%%%%%%%%%%%%%%%%%%%%%%%%%%%%%%%%%%%%%%%%%%%%
\section{\bf Introduction}
%%%%%%%%%%%%%%%%%%%%%%%%%%%%%%%%%%%%%%%%%%%%%%%%%%%%%%%%%%%%%%%%%%%%%%%%%%%%
%%%%%%%%%%%%%%%%%%%%%%%%%%%%%%%%%%%%%%%%%%%%%%%%%%%%%%%%%%%%%%%%%%%%%%%%%%%%
\setcounter{equation}{0}

L. Accardi shows that a Banach space of signed measures is homeomorphic to a Hilbert space 
formed by the so-called square roots of measures \cite{Accardi}. 
The notion of the square roots of measures is introduced to interpret the result of 
Segal \cite{Segal} and Nelson \cite{Nelson} in a context more general than that of Gaussian measures.
In this paper we first discuss about square roots of measures on $\RI$: the countable product space of real lines, and
we redefine the square roots of measures in view of the structure of the projective limit.
The advantage of adopting the new definition is to make available a Fourier transform for the square roots of measures.
It is defined as a translation of the concept of the adjoint measure studied by Yamasaki \cite{Yamasaki85}. 
For a more detailed discussion about the relationship between the ordinary definition and our definition of square
 roots of measures, see\cite{Kakuma}.

A typical square root of a measure is a sequence of square roots of 
density functions of finite dimensional projection of a positive bounded Borel measure $\mu$ on $\RI$.
We will denote it by $\sqrt{\mu}$ and call it a square root of $\mu$.
With this notation it follows easily that the inner product of $\sqrt{\mu}$ and $\sqrt{\nu}$ coincide 
with the Hellinger integral 
\[
H(\mu, \nu) = \int_{\RI} \sqrt{\frac{d\mu}{d\lambda}}\sqrt{\frac{d\nu}{d\lambda}} d\lambda 
\]
where $\lambda$ stands for any positive bounded Borel measure on $\RI$ with respect to which both
$\mu$ and $\nu$ are absolutely continuous. 

We also introduce a notion of the directional differentiation for square roots of measures with the aid of the theory of differentiable measures
studied by Averbuh-Smolyanov-Fomin \cite{Averbuh}, Skorohod \cite{Skorohod}, and many other authors.  
The differentiation for square roots of measures inherit the property of differentiation from ordinary differentiable measures.
In suitable conditions, the Fourier transform is compatible with the differentiation defined here in the following sense:  
if $f$ is a differentiable square root of  a measure in the direction $\rho \in \RI$, 
then the Fourier transform of its directional derivative is given by 
$2\pi \sqrt{-1} \displaystyle \la \xi , \rho \ra \times \Hat{f}$. Here $\la \xi , \rho \ra = \displaystyle \sum_{k=1}^{\I} \xi_k \rho_k$ and $\Hat{f}$ stands for 
the Fourier transform of $f$.

The main purpose of the present paper is to study the L\'{e}vy Laplacian for square roots of measures on 
$C_0[0,T]$: the space of real-valued continuous functions on $[0,T]$ which is $0$ at the origin.
The L\'{e}vy Laplacian is defined as the C\`{e}saro mean of second order differential operators: 
\[
\Delta_L = \lim_{n\to \I} \frac{1}{n} \sum_{k=1}^{n} \frac{\partial^2 }{\partial x^2_n } ,
\]
where $x_1,x_2,\cdots $ constitute a coordinate system of the infinite dimensional vector space under consideration.
There are many results about the L\'{e}vy Laplacian associated with the Gaussian measures and Brownian functionals
(\cite{Hida2}, \cite{Kuo2}, \cite{Obata}).
But in this paper we will be concerned with the measures which are not necessarily of Gaussian type and give a sufficient 
condition for square roots of measures on which the L\'{e}vy Laplacian acts naturally. The key is the Fourier transform 
for square roots of measures.
To transform the L\'{e}vy Laplacian into a multiplication operator by a function makes it easy to examine the domain of the  L\'{e}vy Laplacian.
In fact, it is shown that the function is equal to the quadratic variation of paths on $[0,T]$.
This result will be useful for the theory of Sobolev spaces and pseudo-differential operators of square roots of measures.
However these topics exceed the scope of this paper.

%%%%%%%%%%%%%%%%%%%%%%%%%%%%%%%%%%%%%%%%%%%%%%%%%%%%%%%%%%%%%%%%%%%%%%%%%%%%
%%%%%%%%%%%%%%%%%%%%%%%%%%%%%%%%%%%%%%%%%%%%%%%%%%%%%%%%%%%%%%%%%%%%%%%%%%%%
\section{\bf Definition and Properties of Square Roots of Measures  }
%%%%%%%%%%%%%%%%%%%%%%%%%%%%%%%%%%%%%%%%%%%%%%%%%%%%%%%%%%%%%%%%%%%%%%%%%%%%
%%%%%%%%%%%%%%%%%%%%%%%%%%%%%%%%%%%%%%%%%%%%%%%%%%%%%%%%%%%%%%%%%%%%%%%%%%%%
\setcounter{equation}{0}

Let $\RI$ be countable direct product of real lines
and let $d$ be the distance of $\RI$ defined by
\begin{equation}\label{topology RI}
d(x, y)=\sum_{n=1}^{\I} \frac{1}{2^n} \frac{|x_n - y_n|}{1+|x_n - y_n|}~~(x=\{x_n\}^{\I}_{n=1}, y=\{y_n\}^{\I}_{n=1}\in\RI).
\end{equation}
Then $(\RI,d)$ is a complete separable metric space. Since $\R$ is a nuclear space, so is $\RI$. (see \cite{Bourbaki},
\cite{Schwartz},and \cite{Yamasaki85} for more details).

A sequence of positive $L^2$-functions $\{f_n\},~f_n\in L^2(R^n)$ is called a superprojective system
of $L^2$-functions if    
\begin{equation}\label{weakly L2}
\int_{R}|f_{n+1}(x, x')|^2 dx' \le |f_n(x)|^2~~\mathrm{a.e.}x\in R^n(n=1,2,\cdots).
\end{equation}
The condition \eqref{weakly L2} is called weakly $L^2$-compatibility condition.
If the inequality \eqref{weakly L2} is replaced with an equality, $\{f_n\}$ is called a
projective system of $L^2$-functions and \eqref{weakly L2} is called $L^2$-compatibility condition.
We will denote by $\cLP$ the totality of superprojective systems of $L^2$-functions.
For $\{f_n\}, \{g_n\} \in \cLP$ and $\alpha \in C$ addition and scalar multiplication is defined 
by
\[
\{f_n\}+\{g_n\} = \{f_n + g_n\},  \alpha\{f_n\} = \{\alpha f_n\}
\]
respectively. Then we will denote by $\cL$ the complex linear hull of $\cLP$. 
$\cLP$ is essentially expressed as the sum of four superprojective systems.

\begin{Prop}\label{L2 decomposition}
For all $f \in \cL$ there exist $f^j \in \cLP~(j=1,2,3,4)$ such that 
\begin{equation}\label{L2 dec form}
f=f^1-f^2+\sqrt{-1}(f^3-f^4).
\end{equation}
\end{Prop}

\begin{proof}
Let $f \in \cL$. By definition,  there are
$\alpha_k \in C$ and $f^k \in \cLP$ $(1\le k \le n)$ such that 
$f= \displaystyle \sum_{k=1}^n \alpha_k f^k$. For $\alpha =a+b\sqrt{-1}, ~a,b \in R$ and $f \in \cLP$, we have
\[
\alpha f = (a \vee 0)f - (-a \vee 0)f + \sqrt{-1}(b \vee 0)f - \sqrt{-1}(-b \vee 0)f ,
\]
where $(a \vee 0)=\max\{a,0\}$. This shows $\alpha f$ is of the form such as \eqref{L2 dec form}. 
Thus there are $g^{(j,k)} \in \cLP~(j=1,2,3,4,~1\le k \le n)$ such that $\alpha_k f^k = g^{(1,k)}-g^{(2,k)} +\sqrt{-1} (g^{(3,k)}-g^{(4,k)})$ 
and $f$ is expressed as
\[
\sum_{k=1}^n g^{(1,k)}- \sum_{k=1}^n g^{(2,k)} +\sqrt{-1} \left (\sum_{k=1}^n g^{(3,k)}-\sum_{k=1}^n g^{(4,k)} \right).
\]
Here $\displaystyle \sum_{k=1}^n g^{(j,k)} \in \cLP~(j=1,2,3,4)$. In fact, 
for $f^1=\{f_n^1\}$ and $f^2=\{f_n^2\} \in \cLP$, we have
\begin{align*}
& \int_{R}|f^1_{n+1}(x,x')+ f^2_{n+1}(x,x')|^2 dx'  \\
&=\int_{R}(|f^1_{n+1}(x,x')|^2+|f^2_{n+1}(x,x')|^2+2f^1_{n+1}(x,x')f^2_{n+1}(x,x'))dx', \\
&\le \int_{R}(|f^1_{n+1}(x,x')|^2+|f^2_{n+1}(x,x')|^2)dx'+ 2\sqrt{\int_{R}|f^1_{n+1}(x,x')|^2dx'}\sqrt{\int_{R}|f^2_{n+1}(x,x')|^2dx'}  \\
&\le (f^1_n(x))^2+(f^2_n(x))^2+2f^1_n(x)f^2_n(x) =|f^1_n(x)+f^2_n(x)|^2 .
\end{align*}
This shows that $f^1 + f^2 \in \cLP$. Thus the finite sum of superprojective systems is also a superprojective system.
\end{proof}

For a topological space $X$ we will denote by $\cB(X)$ the totality of Borel sets of $X$.
Let $f=\{f_n\},g=\{g_n\}\in \cL$ and $n\ge 1$. We consider a sequence of $L^1$- functions $\{h^n_k\}$ defined by  
\begin{equation}\label{product L2}
h^n_k(x)=\int_{R^k}f_{n+k}(x, x')g_{n+k}(x, x')dx' ~~(k=1,2,\cdots).
\end{equation}
We first examine $\{h^n_k\}$ if $f, g \in \cLP$. In this case, \eqref{product L2} is a sequence of non-negative $L^1$-functions such that 
$h^n_{k+1}(x) \le h^n_k(x)(k=1,2,\cdots)$. Hence by the monotone convergence theorem, there exists
$h^n \in L^1(\R^n), h_n\ge 0$ such that $\displaystyle\lim_{k\to\I}\int_{\R^n}|h^n_k(x)-h^n(x)| dx =0$.

In addition $\{h_n\}$ satisfies 
\begin{equation}\label{L1 com}
\int_{\R^n}h_{n+1}(x,x')dx'=h_n(x)~(n=1,2,\cdots).  
\end{equation}
The Kolmogorov extension theorem ensures that there exists a bounded positive Borel measure
$\mu$ on $\RI$ such that  
\begin{equation}\label{L1 proj}
\mu(p_n^{-1}(E))=\int_{E} h_n(x) dx
\end{equation}
for all $E \in\cB(\R^n)$. Here $p_n$ stand for the projection from $\RI$ to $\R^n$
defined by 
\[
p_n:\RI \ni (x_1,x_2,\cdots) \mapsto (x_1,x_2,\cdots,x_n).
\]

When $f,g\in\cL$, $\{h^k_n\}$ converges to some $h_n \in L^1(\R^n)$ because
it is a linear combination of monotone decreasing sequences of
$L^1$-functions. Since $\{h_n\}$ is a complex linear 
combination of the sequences like \eqref{L1 com}, there also exists a complex Borel measure
$\mu$ satisfying \eqref{L1 proj}.    
        
The measure $\mu$ defined above is called the product of $f$ and $g$; it is
denoted by $f\cdot g$. In particular we write $|f|^2=f\cdot \bar{f}$ as the square of $f$.
Here $\bar{f}=\{\bar{f}_n\}$ denotes the complex conjugate of $f$.
In this sense the square of $\cL$ is regarded as a measure and the 
element of $\cL$ a square root of a measure on $\RI$. 

\begin{Defi}
Assume that $\mu$ is a positive Borel measure on $\RI$ such that the finite dimensional
projections $p_n(\mu) = \mu(p^{-1}_n(E)), E \in \cB(\R^n)$ is absolutely continuous with the
Lebesgue measure and let $f_n$ be a density function of $p_n(\mu)$. We call $\{\sqrt{f_n}\} \in \cL$
as the square root of $\mu$ and denote it by $\sqrt{\mu}$
\end{Defi}

\begin{Prop}
Fix $f,g \in \cL$ and $\alpha \in\C$. Then the following $(1)-(3)$ hold.
\begin{enumerate}
\item $f\cdot g=g\cdot f$
\item $(f+g)\cdot h=f\cdot h+g\cdot h$
\item $(\alpha f)\cdot g=\alpha (f\cdot g)$
\end{enumerate}
\end{Prop}
 
\begin{Prop}\label{Sc eq}
Let $f,g \in \cL$. 
\begin{equation}\label{Schwarz}
|f\cdot g|(E) \le\sqrt{|f|^2(E)}\sqrt{|g|^2(E)}
\end{equation}
holds for any $E\in\cB(\RI)$. Here $|f \cdot g|$ denotes the total variation of complex Borel measures $f\cdot g$. 
\end{Prop}

\begin{proof}
Let us first prove the case $E$ is of the form $p_n^{-1}(E_n), E_n\in\cB(\R^n)$. 
Let $h^k_n$ and $h_n$ are the function defined by \eqref{product L2}.   
\begin{align*}
|f\cdot g(p_n^{-1}(E_n))|&\le\limsup_{k\to\I}\int_{E_n}\int_{\R^k} |f_{n+k}(x, x')g_{n+k}(x, x')|dx' dx \\
&\le \lim_{k\to\I} \sqrt{\int_{E_n}\int_{\R^k} |f_{n+k}(x, x')|^2dx'dx}\sqrt{\int_{E_n}\int_{\R^k} |g_{n+k}(x, x')|^2dx'dx} \\
&= \sqrt{|f|^2 (p_n^{-1}(E_n))}\sqrt{|g|^2(p_n^{-1}(E_n))}, 
\end{align*}
which proves \eqref{Schwarz}.

If $K$ is a compact set of $\RI$, 
\[
K=\bigcap _{n=1}^{\infty} p_n^{-1}(p_n(K))
\]
holds (see for instance \cite{Bourbaki}). 
Thus we have
\begin{align*}
f\cdot g(K)&= f\cdot g\left( \bigcap _{n=1}^{\infty} p_n^{-1}(p_n(K)) \right) =  \lim_{k\to\I}f\cdot g(p_n^{-1}(p_n(K))) \\
&\le \lim_{k\to\I}\sqrt{|f|^2(p_n^{-1}(p_n(K)))}\sqrt{|g|^2(p_n^{-1}(p_n(K)))}=\sqrt{|f|^2(K)}\sqrt{|g|^2(K)}.
\end{align*}

Since a complex Borel measure on a separable complete metric space is tight (see for instance \cite{DunfordSchwartz}, \cite{Parthasarathy})
and a countable product of separable spaces is also separable,  complex Borel measures on $\RI$ are tight.
So for all $\epsilon >0$ and $E\in\cB(\RI)$ there exists a compact set $K\subset E$ of $\RI$
such that 
\[
\left| |f\cdot g|(E)-|f\cdot g|(K) \right| \le
|f\cdot g| (E\setminus K) \le \epsilon . 
\]

It shows that \eqref{Schwarz} holds for $E\in\cB(\RI)$. 

\end{proof}

Set $f,g \in\cL$. A sesquilinear form $\la \cdot , \cdot \ra_{\cl^2}:\cL\times\cL \mapsto \C$
is defined by $\la f,g \ra_{\cl^2}=(f\cdot \bar{g})(\RI)$ and a seminorm is defined by
$\|f\|_{\cl^2}=\sqrt{\la f,f\ra_{\cl^2}}$.  
$f$ is said to be equivalent to $g$ if $\|f-g\|_{\cl^2}=0$. 
Inequality \eqref{Schwarz} shows that the equivalence relation is well defined.
Let $L^2(\RI)$ be the quotient set of $\cL$ by this equivalence relation. 
$L^2(\RI)$ is a pre-Hilbert space by the inner product $\la\cdot,\cdot\ra_{L^2}$ induced by
$\la\cdot,\cdot\ra_{\cl^2}$. Here we write $\|f\|_{L^2}=\sqrt{\la f,f\ra_{L^2}}$.
We will show that $L^2(\RI)$ is a Hilbert space.  
The following three lemmas are needed to prove. 

\begin{Lem}\label{com cL} 
For all $f=\{f_n\} \in \cL, f_n(x) \ge 0$, there exists $g=\{g_n\}$ which is equivalent to $f$ and  
satisfy the $L^2$-compatibility condition i.e.
\[
\int_{\R}|g_{n+1}(x,x')|^2dx'=|g_n(x)|^2~~a.e.x\in\R^n.
\]
\end{Lem}

\begin{proof}
Assume that $f$ is of the form $f=f^1-f^2, f^1=\{f_n^1\}, f^2=\{f_n^2\} \in \cLP,\\ f_n^1-f_n^2 \ge 0$. 
Set $g_n(x)=\displaystyle \sqrt{\lim_{k\to\I}\int_{\R^k} |f_{n+k}(x,x')|^2dx'}$. 
It is easy to check that $g=\{g_n\}$ satisfies $L^2$-compatibility condition. $f\sim g$ is shown by 
\begin{align*}
\|f-g\|^2_{\mathcal{L}^2}&=\lim_{n\to\I}\int_{\R^n}|f_n(x)-g_n(x)|^2 dx \\
&\le \lim_{n\to\I} \int_{\R^n}\left||f_n(x)|^2-|g_n(x)|^2 \right|dx \\
& = \lim_{n\to\I} \int_{\R^n}\left||f_n^1(x)-f_n^2(x)|^2-\lim_{k\to\I}\int_{\R^k}|f^1_{n+k}(x,x')-f^2_{n+k}(x,x')|^2dx' \right|dx \\
&\le \sum_{j=1}^{2}\lim_{n\to\I}\int_{\R^n}\left(|f_n^j(x)|^2-\lim_{k\to\I}\int_{\R^k}|f_{n+k}^j(x,x')|^2dx'\right)dx \\
&+2\lim_{n\to\I}\int_{\R^n}\left(f_n^1(x)f^2_n(x)-\lim_{k\to\I}\int_{\R^k}f_{n+k}^1(x,x')f_{n+k}^2(x,x')dx'\right)dx=0 . 
\end{align*}
Here we have used 
\[
|a-b|^2 \le |a^2-b^2| ~~(a, b\ge 0).
\]
\end{proof}

\begin{Lem}\label{sup cL}
Let $f=\{f_n\}\in\cL$. Then there exist $h^j=\{h^j_n\}\in\cL, \\ h_n^j(x)\ge 0~ (j=1,2,3,4)$ such that
\begin{enumerate}
\item $h=h^1-h^2+\sqrt{-1}(h^3-h^4) \in\cL$ is equivalent to $f$, 
\item $h_n^1(x)h_n^2(x)=h_n^3(x)h_n^4(x)=0, ~\mathrm{a.e.} x\in\R^n$.
\end{enumerate}
\end{Lem}

\begin{proof}
By Proposition \ref{L2 decomposition}, there exist $\{f^j_n\}\in\cLP~(j=1,2,3,4)$ such 
that $f=f^1-f^2+\sqrt{-1}(f^3-f^4)$. And lemma \ref{com cL} shows that there exist 
$g^j=\{g_n^j\} \in\cLP ~(j=1,2,3,4)$ such that 
\[
g^j\sim f^j,~~ \int_{\R}|g_{n+1}^j(x,x')|^2dx'=|g_n^j(x)|^2~~\mathrm{a.e.}x\in\R^n~ (j=1,2,3,4).
\]
Now set 
\[
H_n(x)=g_n^1(x)+g_n^2(x)-|g^1_n(x)-g^2_n(x)|.
\]
Due to
\begin{align*}
|H_n(x)|^2&=|g^1_n(x)+g^2_n(x)|^2+|g^1_n(x)-g^2_n(x)|^2-2\bigl||g_n^1(x)|^2-|g_n^2(x)|^2\bigr| \\
&=2(|g_n^1(x)|^2+|g_n^2(x)|^2)-2\bigl||g_n^1(x)|^2-|g_n^2(x)|^2\bigr|,
\end{align*}
it follows that
\begin{align*}
\int_{\R} & |H_{n+1}(x,x')|^2dx' \\
&=2\int_{\R}(|g_{n+1}^1(x,x')|^2+|g_{n+1}^2(x,x')|^2)dx'-2\int_{\R}\bigl| |g_{n+1}^1(x,x')|^2-|g_{n+1}^2(x,x')|^2\bigr|dx' \\
&\le 2(|g_n^1(x)|^2+|g_n^2(x)|^2)-2\left|\int_{\R}(|g_{n+1}^1(x,x')|^2-|g_{n+1}^2(x,x')|^2)dx'
\right| \\
&=2(|g_n^1(x)|^2+|g_n^2(x)|^2)-2\bigl| |g_n^1(x)|^2-|g_n^2(x)|^2\bigr|=|H_n(x)|^2.
\end{align*}
This shows that $\{H_n\}$ is the superprojective system. Thus  
$\displaystyle \frac{|g_n^1-g_n^2|\pm (g_n^1-g^2_n)}{2}$
belongs to $\cL$. In the same way, we can show that $\displaystyle \frac{|g_n^3-g_n^4|\pm (g_n^3-g^4_n)}{2}$ belongs to $\cL$. 
Here by letting
\begin{align*}
h^1&= \frac{|g_n^1-g_n^2|+ (g_n^1-g^2_n)}{2} ,~h^2=\frac{|g_n^1-g_n^2| - (g_n^1-g^2_n)}{2} , \\
h^3&= \frac{|g_n^3-g_n^4|+ (g_n^3-g^4_n)}{2} ,~h^4=\frac{|g_n^3-g_n^4| - (g_n^3-g^4_n)}{2} ,
\end{align*}
we can show that both of $(1)$ and $(2)$ are satisfied.
\end{proof}

For a topological space $X$, let $M(X)$ be the totality of complex Borel measures. $M(X)$ is a complete metric space with the total variation norm 
$\displaystyle \|\mu\|= \sup_{E \in \cB(X)} |\mu(E)|$.
\begin{Lem}\label{complete L^1}
Let $M_0(\RI)$ be a collection of complex Borel measures whose
 finite dimensional projections $p_n(\mu)=\mu(p_n^{-1}(E_n))~(E_n\in\cB(\R^n) )$ are absolutely continuous with respect to 
the Lebesgue measure on $\R^n$. Then $M_0(\R^{\I})$ is a closed subspace of $M(\RI)$.
\end{Lem}

\begin{proof}
Let $\{\mu^j\}$ be a Cauchy sequence of $M_0(\R^{\I})$. Since the totality of complex Borel measures on metric spaces are complete 
with the total variation norm (see for instance \cite{Parthasarathy}), ${\mu^j}$ converges to $\mu$: a complex measure on $\RI$.
Let $g^j_n$ be the density function of $p_n(\mu_j)$. Then 
\[
\int_{R^n}|g^j_n(x)-g^l_n(x)|dx \le \|\mu^j-\mu^l \|.
\]  
Therefore for all $n$, $g^j_n$ is a Cauchy sequence of $L^1(R^n)$. This shows that the density function of $p_n(\mu)$
is $\displaystyle \lim_{n\to\I} g_n^j(x)$.
  
\end{proof}

\begin{Th}
$L^2(\RI)$ is a complete metric space with respect to $\|\cdot\|_{L^2}$.
\end{Th}

\begin{proof}
Let $f^l=\{f_n^l\}$ is a Cauchy sequence with respect to $\|\cdot\|_{L^2}$. 
Lemma \ref{sup cL} ensures that there exists the representative element of $f^l$
which is of the form $h^{l,1}-h^{l,2}+i(h^{l,3}-h^{l,4})$ that $h^{l,1}=\{h_n^{l,j}\}, h_n^{l,j}(x)\ge 0~(j=1,2,3,4)$ belong to $\cL$
and $h_n^{l,1}(x)h_n^{l,2}(x)=h_n^{l,3}(x)h_n^{l,4}(x)=0$ for each $l \ge 1$.
For all $l,m\in\N$,
\begin{align*}
\|f^l- f^m\|^2_{L^2} & =\lim_{n\to\I}\int_{\R^n}|(h_n^{l,1}-h_n^{l,2})-i(h_n^{l,3}-h_n^{l,4}) \\
& -(h_n^{m,1}-h_n^{m,2})+i(h_n^{m,3}-h_n^{m,4})|^2dx \ge \sum_{j=1}^{4}\lim_{n\to\I}\int_{\R^n}|h_n^{l,j}-h_n^{m,j}|^2dx.
\end{align*}
From this it follows that $\{[h^{l,j}]\}~(j=1,2,3,4)$ are Cauchy sequences with respect to $\|\cdot\|_{L^2}$. Here $[f]$ denote the equivalence class to which $f$ belongs.
Hence by Lemma \ref{com cL}, $\{[h^{l,j}]\}$ is equivalent to a projective system of $L^2$-functions. 
Thus we can assume that $f^l$ is an equivalence class to
which a projective system $g^l=\{g^l_n\}$ belongs for all $l\in\N$ without loss of generality.   
  
If $E$ is of the form $p_n^{-1}(E_n)(E_n\in\cB(\R^n))$, we obtain
\begin{align*}
\bigl||g^l|^2 & (E) - |g^m|^2(E)\bigr| \\
&=\left|\lim_{k\to\I}\left(
\int_{E_n}\int_{\R^k}|g_{n+k}^l(x,x')|^2dxdx'-\int_{E_n}\int_{\R^k}|g_{n+k}^m(x,x')|^2dxdx'
\right)\right| \\
&=\lim_{k\to\I}\left(\sqrt{
\int_{E_n}\int_{\R^k}|g_{n+k}^l(x,x')|^2dxdx'
}+\sqrt{
\int_{E_n}\int_{\R^k}|g_{n+k}^m(x,x')|^2dxdx'
}\right) \\
&~~~~~~~~~~~~~~~~~\times\left|\sqrt{
\int_{E_n}\int_{\R^k}|g_{n+k}^l(x,x')|^2dxdx'
}-\sqrt{
\int_{E_n}\int_{\R^k}|g_{n+k}^m(x,x')|^2dxdx'
}\right| \\
&\le(\|g^l\|_{L^2}+\|g^m\|_{L^2})\lim_{k\to\I}\sqrt{
\int_{E_n}\int_{\R^k}|g_{n+k}^l(x,x')-g_{n+k}^m(x,x')|^2dxdx'
} \\
&=(\|g^l\|_{L^2}+\|g^m\|_{L^2})\|g^l-g^m\|_{L^2}
\end{align*}
by using the triangle inequality $\left|\|f\|_{L^2}-\|g\|_{L^2} \right|\le \|f-g\|_{L^2}.$
Provided that $\|g^l\|_{L^2}$ is bounded by $M/2$, we can conclude that 
\[
\bigl||g^l|^2(E)-|g^m|^2(E)\bigr| \le M\|g^l-g^m\|_{L^2}.
\]
This inequality holds for all $E\in\cB(\RI)$ as in the proof of Proposition\ref{Sc eq}. Thus we obtain 
\[
\||g^l|^2-|g^m|^2\|_{M} \le M\|g^l-g^m\|_{L^2}.
\]
 According to the above inequality $|g^l|^2$ is a Cauchy sequence with respect to $\|\cdot\|$. 
From Lemma \ref{complete L^1} there exist a positive Borel measure $\mu$ on $\RI$ as the limit of $|g^l|^2$ such that 
\[
\mu_n(E_n)=\int_{E_n}G_n(x)dx
\] 
for $G_n\in L^1(\R^n), G_n(x) \ge 0$.

Set $g'=\{\sqrt{G_n}\}$ and $H_n(x)=\displaystyle\lim_{k\to\I}\int_{R^k}|g_{n+k}(x,x')|^2dx'$. 
We see at once \\ $g'\in \cLP$. In addition   
\begin{align*}
\|g^l-[g']\|^2_{L^2}&=\lim_{n\to\I}\int_{\R^n}|g_n^l(x)-\sqrt{G_n(x)}|^2dx 
\le \lim_{n\to\I}\int_{\R^n}\bigl||g^l_n(x)|^2-G_n(x)\bigr|dx \\
&\le\lim_{n\to\I}\int_{\R^n}|
|g_n^l(x)|^2-H_n(x)
|dx +\lim_{n\to\I}\int_{\R^n}|
H_n(x)-G_n(x)
|dx \\
&\le \|\mu-|g^l|^2\|.
\end{align*}
Thus we obtain $\displaystyle\lim_{l\to\I}\|g^l-[g']\|_{\mathcal{L}^2}=0$, which proves the theorem. 

\end{proof}

%From above discussions we are able to show that there exists an homeomorphism from $M_0(\RI)$ to $L^2(\RI)$.
%This construction was motivated by \cite{Accardi}.  

%\begin{Prop}
%Let $\mu \in M(\RI)$ whose the Jordan decomposition is given by $\mu_1-\mu_2+i(\mu_3-\mu_4)$.
%$\alpha(\mu) = \sqrt{\mu_1}-\sqrt{\mu_2}+\sqrt{-1}(\sqrt{\mu_3}-\sqrt{\mu_4})$ is a homeomorphism 
%from $M_0(\RI)$ to $L^2(\RI)$ enjoying the following properties:
%\begin{enumerate}
%\item $\|\mu\|= \|\alpha(\mu)\|^2_{L^2}$ for all $\mu \in M_0(\RI)$,
%\item $<\alpha(\mu),\alpha(\nu)>_{L^2} = 0$ if and only if $\mu$ and $\nu$ are
%mutually singular for all $\mu, \nu \in M_0(\RI), \\ \mu \ge 0, \nu \ge 0$,
%\item $\alpha(\mu + \nu) = \alpha(\mu) + \alpha(\nu)$ if and only if $\mu$ and $\nu$ are
%mutually singular for all $\mu, \nu \in M_0(\RI)$,
%\item $\alpha(\lambda \mu) = \sqrt{\lambda} \alpha(\mu)$ for all $\lambda \ge 0$ and $\mu \in M_0(\RI)$,
%\item $\alpha(-\mu)=-\alpha(\mu)$ for all $\mu \in M_0(\RI)$,
%\end{enumerate}
%\end{Prop}

Let us introduce two important operators on $L^2(\RI)$: multiplication operators and  convolution products.  
Suppose that $w \in L^2(\RI, \cB(\RI),|f|^2)$ and $f \in L^2(\RI)$. 
For arbitrary $g \in L^2(\RI)$, let $T: L^2(\RI) \to C$ be a linear functional defined by
\[
T(g) = \int_{\RI} w(x) (f\cdot g)(dx) .
\] 
From \eqref{Sc eq} we have 
\[
|T(g)| \le \displaystyle  \left( \int_{\RI} |w(x)|^2 |f|^2(dx)  \right)^{\frac{1}{2}} \|g\| .
\]
So $T$ is bounded
and there exists $h \in L^2(\RI)$ such that $T(g) = \la h, g \ra$ because of Riesz representation theorem. 
We call $h$ mentioned above the multiplication of $w$ and $f$,  and write $w \times f$.

\begin{Prop}
Let $f, g \in L^2(\RI)$ and $u, v$ are bounded Borel measurable complex-valued functions on $\RI$. 
Then, the following $(1)-(5)$ hold. 
\begin{enumerate}
\item $(u+v)\times f=u\times f+v\times f$
\item $u\times (f+g)=u\times f+u\times g$
\item $\overline{u\times f}=\bar{u}\times\bar{f}$
\item $(u\times f,g)=(f,w\times g)$
\item $u\times (v\times f)=(u\times v)\times f$
\end{enumerate}

\end{Prop}

\begin{Lem}\label{Lemconv}
Suppose that $f=\{f_n\}_{n=1}^{\infty} \in \cLP$ and $\mu \in M_0(\RI)$ be a positive measure such that the density function of $p_n(\mu)$ is given by $g_n \ge 0$ for all $n$.
Then $\{ f_n \ast g_n \} \in \cLP$ hold. Here $\displaystyle (f_n \ast g_n)(x) = \int_{\R^n} f_n(x-y) g_n(y) dy$. 
\end{Lem}

\begin{proof}
It suffices to show that $\{ f_n \ast g_n \}$ satisfies \eqref{weakly L2}. Hence by
\begin{align*}
\int_{\R}\biggl(\int_{\R} & |f_{n+1} (x-y,x'-y')|^2 g_{n+1}(y,y')dy'\biggr) dx' ~~~~~~~~~~~~~~~~~~~~~~~~~~~~~~~~~~~~~~~~~~~~~~~~~  \\
&=\int_{\R}\left(\int_{\R}|f_{n+1}(x-y,x'-y')|^2 dx'\right) g_{n+1}(y,y')dy'   \\
&\le |f_n(x-y)|^2 \int_{\R}g_{n+1}(y,y')dy'=|f_n(x-y)|^2 g_n(y), 
\end{align*}
We have
\begin{align*}
\int_{\R}|(f_{n+1}&\ast g_{n+1})(x,x')|^2 dx' \\
&=\int_{\R}\left(\int_{\R^{n+1}}f_{n+1}(x-y,x'-y')g_{n+1}(y,y')dydy'\right)^2dx' \\
&\le\int_{\R}\left(\int_{\R^n}\sqrt{\int_{\R}|f_{n+1}(x-y,x'-y')|^2 g_{n+1}(y,y')dy'}\sqrt{\int_{\R}g_{n+1}(y,y')dy'}dy\right)^2dx' \\
&=\int_{\R}\left(\int_{\R^n\times\R^n}\sqrt{\int_{\R}|f_{n+1}(x-s,x'-y')|^2 g_{n+1}(s,y')dy'}\sqrt{g_n(s)}\right. \\
&~~~~~~~~~\left.\times\sqrt{\int_{\R}|f_{n+1}(x-t,x'-y')|^2 g_{n+1}(t,y')dy'}\sqrt{g_n(t)}dsdt\right) dx' \\
&=\int_{\R^n\times\R^n}\left(\int_{\R}\sqrt{\int_{\R}|f_{n+1}(x-s,x'-y')|^2 g_{n+1}(s,y')dy'}\sqrt{g_n(s)}\right. \\
&~~~~~~~~~\left.\times\sqrt{\int_{\R}|f_{n+1}(x-t,x'-y')|^2 g_{n+1}(t,y')dy'}\sqrt{g_n(t)}dx' \right) dsdt \\
&\le\int_{\R^n\times\R^n}\sqrt{\int_{\R}\left(\int_{\R}|f_{n+1}(x-s,x'-y')|^2 g_{n+1}(s,y')\right)dx'} \\
&~~~~~~~~~\times\sqrt{\int_{\R}\left(\int_{\R}|f_{n+1}(x-t,x'-y')|^2g_{n+1}(t,y')\right)dx'}\sqrt{g_n(s)g_n(t)}dsdt \\
&\le\int_{\R^n\times\R^n}\sqrt{|f_n(x-s)|^2 g_n(s)}\sqrt{|f_n(x-t)|^2 g_n(t)}\sqrt{g_n(s)g_n(t)}dsdt \\
&=|(f_n\ast g_n)(x)|^2,
\end{align*}
which proves the lemma. 

\end{proof}

For arbitrary $f=\{f_n\} \in \cL$ and $\mu \in M_0(\RI)$ such that the density function of $p_n(\mu)$ is given by $g_n$ for all $n$, 
the convolution product of $\mu$ and $f$, denoted by $\mu \ast f$, is defined to be $\{ f_n \ast g_n\}$.
Hence by Lemma \ref{Lemconv}, $\mu \ast f$ can be expressed as the complex linear combination of $\cLP$. It follows that $\mu\ast f \in \cL$.
In the case $f\in L^2(\RI)$, choose $\{f_n\}\in\cL$ a representative of $f$ and the convolution product of $\mu$ and $f$ is defined by $[\{f_n \ast g_n\}]$.
Since by Young's inequality, $\| \mu_n \ast f_n \| \le \|\mu_n\| \|f_n\|$. By taking the limit of this inequality, we have  
$\| \mu \ast f \| \le \|\mu\| \|f\|$. Thus the definition of $\mu \ast f$ is independent of the choice of the representative.

%%%%%%%%%%%%%%%%%%%%%%%%%%%%%%%%%%%%%%%%%%%%%%%%%%%%%%%%%%%%%%%%%%%%%%%%%%%%
%%%%%%%%%%%%%%%%%%%%%%%%%%%%%%%%%%%%%%%%%%%%%%%%%%%%%%%%%%%%%%%%%%%%%%%%%%%%
\section{\bf Fourier transform on $L^2(\RI)$}
%%%%%%%%%%%%%%%%%%%%%%%%%%%%%%%%%%%%%%%%%%%%%%%%%%%%%%%%%%%%%%%%%%%%%%%%%%%%
%%%%%%%%%%%%%%%%%%%%%%%%%%%%%%%%%%%%%%%%%%%%%%%%%%%%%%%%%%%%%%%%%%%%%%%%%%%%
\setcounter{equation}{0}
 
Let $\RIO = \{ (x_1, x_2, \cdots): \mathrm{~there~exists~} N, x_n=0~~\mathrm{for~all~} n>N \}.$
The topology of $\RIO$ is defined by seminorms $\{p_x\}_{x\in\RI}$:
$\displaystyle p_x(y) = \left| \sum_{n=1}^{\I} x_n y_n \right| ~~(y \in \RIO).$
Since there exists $N$ such that $y_n=0$ for $n>N$, the right hand side makes sense. $\RIO$ is homeomorphic to the 
countable direct sum of $\R$.  

Let $a=\{a_n\}$ be a positive sequence and $H_a$ be a Hilbert space defined by \\
$\displaystyle H_a=\{ x =\{x_n\}_{n=1}^{\infty} ; \sum_{n=1}^{\infty} a^2_n x^2_n < \infty \}$
equipped with the inner product $\displaystyle \langle x,y \rangle_a = \sum_{n=1}^{\infty} a^2_n x_n y_n$ and the norm 
$\|x\|_a = \sqrt{\la x,x \ra_a}$ . 

The following assertions are special cases of Minlos' theorem and Sazanov' theorem. 
For a thorough treatment we refer the reader to \cite{Bourbaki}, \cite{Yamasaki85}. 
\begin{Th}[Minlos]\label{Minlos}
For all continuous function $\phi$ on $\RIO$ satisfying
\begin{equation}\label{pos def}
\sum_{i,j=1}^{n} \alpha_i \bar{\alpha_j} \phi(x_i-x_j) \ge 0
\end{equation}
for all $\alpha_1,\cdots,\alpha_n\in\C, x_1,\cdots,x_n\in\RIO(n=1,2,\cdots)$,
there exists a bounded positive Borel measure $\mu$ on $\RI$ such that 
\[
\phi(x)=\int_{\RI}e^{
2\pi \sqrt{-1} \la x,y \ra
}\mu(dy),\la x,y \ra=\sum_{n=1}^{\I} x_n y_n .  
\]
\end{Th}

Such $\phi$ is called the characteristic function of $\mu$, and denoted by $\Hat{\mu}$.
A function satisfying \eqref{pos def} is said to be positive definite. This theorem also holds for 
a continuous positive definite function on a nuclear space $S \subset \RI$ as the characteristic function of the measure on its dual space $S^{'}$. We will
establish the relation between the inner product of $L^2(\RI)$ and 
positive definite functions on $\RI$. 

\begin{Th}[Sazanov]\label{Sazanov-Minlos}
Suppose that $a=\{a_n\}$ and $b=\{b_n\}$ are positive sequences satisfying \\
$\displaystyle \sum_{n=1}^{\infty} a^2_n b^2_n < \infty$. 
For all continuous positive definite function $\phi$ on $H_b$, there exists a positive bounded Borel measure $\mu$ on $H_a$
whose characteristic function is given by $\phi$. 
\end{Th}
     
Set $f=\{f_n\}\in \cL$ and $h=(h_1,h_2,\cdots)\in\RI$. A translation operator $\tau_h$ on $\cL$ is defined by
$\displaystyle \tau_h f =\{ f_n(x_1-h_1,\cdots,x_n-h_n)\}$ .
By definition, it follows that 
\begin{equation}\label{shift inv}
\|\tau_h f\|_{\cl^2}=\|f\|_{\cl^2}.
\end{equation}

Let $E\subset\RI$ be a complete metric space.  
$f\in \cL$ is said to be $E$-shift continuous if 
$\displaystyle \lim_{n \to \I} \|\tau_{h_n} f - f \|_{\cl^2}=0$ 
whenever a sequence $\{h_n\}$ in $E$ converges to limit $0$.
Let $S\subset\RI$ be a nuclear space.
$f\in \cL$ is said to be $S$-shift continuous if $\displaystyle \lim_{\alpha} \|\tau_{h_{\alpha}} f - f \|_{\cl^2}=0$
whenever a net $\{ h_{\alpha} \}$ in $S$ converges to $0$.

A translation operator $\tau_h$ on $L^2(\RI)$ is defined by 
$\displaystyle\tau_h [f] = [\tau_h f] (f\in \cL).$
Equation \eqref{shift inv} makes this definition possible. Shift continuity of $L^2(\RI)$ is defined as well as $\cL$. 
We will denote by $\LL$ the totality of $\RIO$-shift continuous elements of $L^2(\RI)$. \\
\\

\begin{Exam}
If a bounded positive Borel measure $\mu$ is of the form
\[
\mu=\prod_{n=1}^{\I} f_n dx_n,~~ f_n \in L^1(\R^n), ~~f_n \ge 0, ~~\int_{\R}f_n(x)dx=1,
\]
then the square root of $\mu$ is $\RIO$-shift continuous. 
In fact, letting $g_n=\mathcal{F}(\sqrt{f_n})$, $\displaystyle \int_{\R}|g_n(x)|^2dx=1$ follows by Parseval's equation. Thus 
$\displaystyle \lambda=\prod_{n=1}^{\I} |g_n|^2 dx_n$ also defines a bounded positive Borel measure on $\RI$.
We obtain
\begin{align*}
\|\tau_h &\sqrt{\mu} - \sqrt{\mu}\|_{\cl^2}^2 
=\int_{\R^n}\left|\sqrt{f_1(x-h_1)\cdots f_n(x-h_n)}-\sqrt{f_1(x)\cdots f_n(x)}\right|dx_1\cdots dx_n \\
&=\int_{\R^n}|e^{2\pi \sqrt{-1}\la \xi,h \ra_n}-1|^2 |g_1(\xi)|^2\cdots |g_n(\xi)|^2 d\xi_1\cdots d\xi_n \\
&=\int_{\RI}|e^{2\pi \sqrt{-1}\la \xi,h \ra}-1|^2 \lambda(d\xi).
\end{align*}
Here $h=(h_1,\cdots,h_n,0,\cdots)$. Because the characteristic function on $\RI$ is continuous 
function on $\RIO$, so is $\sqrt{\mu}$. 
\end{Exam}

\begin{Defi}
Assume that $f=\{f_n\} \in \cL$ is $\RIO$-shift continuous. Then, a Fourier transform of $f$ is defined by $\{\Hat{f}_n\}$ and denoted by $\Hat{f}$ or $\mathcal{F} f$
Here, 
\[
\Hat{f}_n(\xi)=\LIM_{L\to\I}\int_{|x|_n\le L} e^{ 2\pi \sqrt{-1}\la x, \xi \ra_n} f(x)dx, \\
\la x,\xi \ra_n =\sum_{j=1}^{n}x_j \xi_j, |x|_n=\sqrt{\la x,x \ra_n}. 
\] 
\end{Defi}

The sequence $\{\Hat{f}_n\}$ is not always included in $\cL$. However, it is possible to give a definition of
the square of $\Hat{f}=\{\Hat{f}_n\}$ as a bounded positive Borel measure.

\begin{Prop}\label{weak com}
Assume that $f=\{f_n\} \in \cL$ is $\RIO$-shift continuous. Then, there exists a 
bounded positive Borel measure $\mu$ on $\RI$ such that
\begin{equation}\label{weak * com}
\lim_{n\to\I}\int_{\R^n}e^{ 2\pi \sqrt{-1}\la x,\xi \ra_n} |\Hat{f}_n(\xi)|^2d\xi
=\int_{\RI}e^{ 
2\pi \sqrt{-1} \la x,y \ra
}\mu(dy)
\end{equation}
for all $x \in \RIO$.
\end{Prop}

\begin{proof}

Set $\phi(x)=\la \tau_x f,f \ra_{\cl^2}(x\in \RIO)$. By assumption, $\phi$ is continuous function on $\RIO$. And \eqref{shift inv} ensures that     
\[
\sum_{i,j=1}^{n}\alpha_i \bar{\alpha_j} \phi(x_i-x_j) 
=\sum_{i,j=1}^{n}\alpha_i \bar{\alpha_j} \la \tau_{x_i}f, \tau_{x_j}f \ra_{\cl^2}
=\|\sum_{i=1}^{n}\alpha_i \tau_{x_i} f\|_{\cl^2} \ge 0.
\]
for all $\alpha_1,\cdots,\alpha_n\in\C, x_1,\cdots,x_n\in\RIO(n=1,2,\cdots)$.
Thus $\phi$ is positive definite. 
From Minlos' theorem \ref{Minlos}, there exists a bounded positive measure on $\RI$
such that 
\[
\la \tau_x f,f \ra_{\cl^2} = \int_{\RI}e^{
2\pi \sqrt{-1} \la x,y \ra
}\mu(dy).
\]
By definition, the left hand side of this equation is equal to
\begin{align*}
& \lim_{n\to\I} \int_{\R^n} f_n(y_1-x_1,\cdots,y_n-x_n) f_n(y_1,\cdots,y_n) dy_1\cdots dy_n \\
&=\lim_{n\to\I}\int_{\R^n}e^{ 2\pi \sqrt{-1} \la x, y \ra_n}|\Hat{f}_n(y)|^2 dy, 
\end{align*}
which proves the proposition. 
\end{proof}

Let $\mu^k_n$ be a bounded positive Borel measure on $\R^n$ defined by  
\[
\mu^k_n(E)=\int_{E}\int_{\R^k} |\widehat{f}_{n+k}(\xi, \xi')|^2d\xi'd\xi
\]
for $E\in\cB(\R^n)$, and $\mu_n$ be a bounded positive Borel measure on $\R^n$
defined by 
\begin{equation}\label{mar meas}
\mu_n(E)=\mu(p_n^{-1}(E)).   
\end{equation}

Substituting $x=(x_1,\cdots,x_n,0,\cdots)\in \RIO$ into \eqref{weak * com} we can assert that
the characteristic function of $\mu_n^k$ converges to that of $\mu_n$ at each point. Thus $\mu_n^k$ converges weakly to 
$\mu_n$ i.e. 
\[
\lim_{k\to\I}\int_{\R^n}f(x)\mu_n^k(dx) = \int_{\R^n}f(x)\mu_n(dx)
\]
for all continuous bounded functions $f$ and $n=1,2,\cdots$.  

Let $\tilde{\mu}^k$ be a bounded positive Borel measure on $\RI$ such that 
\[
\tilde{\mu}^k(p^{-1}_j(E))=\begin{cases}
                   \mu^k_j(E) & \text{$j\le k$}  \\
                   \displaystyle\int_E |\widehat{f}_k(\xi)|^2\delta(\xi')d\xi d\xi' & \text{$j>k$}.
                   \end{cases}
\]
Here $\delta$ means the Dirac delta measure on $\R^{j-k}$.  

At $\xi=(\xi_1,\cdots,\xi_j,0,\cdots)\in\RIO(j\le k)$, it follows that
\[
\Hat{\tilde{\mu}}^k(\xi_1,\cdots,\xi_j,0,\cdots)=\Hat{\mu}_j^k(\xi_1,\cdots,\xi_j).
\]
Thus the characteristic function of $\tilde{\mu}^k$ converges to that of $\mu$
at each point. Since $\RI$ is nuclear space, it follows that $\tilde{\mu}^k$ converges weakly to $\mu$ (see \cite{Bourbaki} for more details).

\begin{Prop}
The measure $\mu$ defined in Proposition \ref{weak com} satisfies 
\begin{equation}\label{con meas}
\lim_{\alpha}\|\tau_{h_{\alpha}}\mu-\mu\|=0
\end{equation}
whenever a net $\{h_\alpha\}$ in $\RIO$ converges to limit $0$.
\end{Prop}

We call a measure satisfying \eqref{con meas} the $\RIO$-shift continuous measure.

\begin{proof}
Fix $h\in\RIO$. Since $\tau_h\tilde{\mu}^k-\tilde{\mu}^k$ converges weakly to
$\tau_h\mu-\mu$, it follows that
\begin{equation}\label{Var}
\| \tau_h\mu-\mu \| \le \liminf_{k\to\I} \| \tau_h \tilde{\mu}^k-\tilde{\mu}^k \| 
\end{equation}
(see for instance \cite{Varadarajan}).
By definition, it follows that
\[
\| \tau_h \tilde{\mu}^k-\tilde{\mu}^k \|=\int_{\R^k} 
\left||\Hat{f}_k(\xi_1-h_1,\cdots,\xi_k-h_k)|^2-|\Hat{f}_k(\xi_1,\cdots,x_k)|^2\right|
d\xi_1\cdots d\xi_k .
\]
Since 
\[
\int_{\R^n}\left||f(x)|^2-|g(x)|^2\right|dx \le (\|f\|_{L^2(\R^n)}+\|g\|_{L^2(\R^n)})\|f-g\|_{L^2(\R^n)}
\]
for all $f, g\in L^2(\R^n)$, it follows that 
\begin{align*}
& \|\tau_h \tilde{\mu}^k-\tilde{\mu}^k \| \\ 
&\le 2\sqrt{\int_{\R^k}|\Hat{f}_k(\xi)|^2d\xi} \sqrt{\int_{\R^k}|\Hat{f}_k(\xi_1-h_1,\cdots,\xi_k-h_k)-\Hat{f}_k(\xi_1,\cdots,\xi_k)|^2d\xi_1\cdots d\xi_k} \\
&=2\sqrt{
\int_{\R^k}|f_k(x)|^2dx
}\sqrt{
\int_{\R^k}|(1-e^{
2\pi \sqrt{-1}\la x, h\ra_k 
}
f_k(x)|^2dx.
}
\end{align*}
At $h=(h_1,\cdots,h_j,0,\cdots)\in\RIO$, it follows that

\begin{align*}
\lim_{k\to\I} \int_{\R^k} & |(1- e^{2\pi \sqrt{-1}\la x, h\ra_k} f_k(x)|^2dx \\
&=2\lim_{k\to\I}\int_{\R^n}(1-\cos 2\pi \la x, h\ra_n)\int_{\R^k}|f_{n+k}(x, x')|^2dx'dx \\
&=2\int_{\RI}(1-\cos 2\pi \la x, h\ra)|f|^2(dx).
\end{align*}

Thus \eqref{Var} is bounded by 
\begin{equation}\label{real part cha}
4\|f\|_{L^2} \int_{\RI}(1-\cos 2\pi \la x, h\ra)|f|^2(dx)
\end{equation}
at each $y\in\RIO$.
Since \eqref{real part cha} is the real part of the characteristic function of the square of $f$,
it is continuous function on $\RIO$, which prove the theorem. 
\end{proof}

Proposition \ref{weak com} implies that there exists a sequence of 
$L^1$-functions $\{h_n\}$ which satisfies \eqref{L1 com}, and
\begin{equation}\label{weak* con}
\lim_{k\to\I}\int_{\R^n}g(x)\left(
\int_{\R^k}|\Hat{f}_{n+k}(x,x')|^2dx'
\right)dx
=\int_{\R^n}g(x)h_n(x)dx
\end{equation}
for all bounded continuous functions $g$ on $\R^n$.

A deeper discussion make it possible to improve the convergence of \eqref{weak* con}.
In fact, it follows that
\begin{equation}\label{L1 conv}
\lim_{k\to\I}\int_{\R^n}
\left|
h_n(x)-\int_{\R^k}|\Hat{f}_{n+k}(x,x')|^2dx'
\right|dx=0.
\end{equation}
for $n=1,2,\cdots$.
Moreover, \eqref{L1 conv} is equivalent to the $\RIO$ shift continuity of $f\in\cL$.
It shows that a totality of square roots of measures satisfying \eqref{L1 conv} forms 
a linear subspace of $\cL$ and the squares of such ones are $\RIO$-shift continuous measures. A detailed proof will appear in elsewhere.

We call the measure defined by Proposition \ref{weak com} the square of $\Hat{f}$.
For $f,g\in\cL$, the product of $\Hat{f}$ and $\bar{\Hat{g}}$ is defined by 
\[
\frac{|\Hat{f}+\Hat{g}|^2-|\Hat{f}-\Hat{g}|^2+\sqrt{-1}|\Hat{f}+\sqrt{-1}\Hat{g}|^2-
\sqrt{-1}|\Hat{f}-\sqrt{-1}\Hat{g}|^2
}{4}
\]
and denoted by $\Hat{f}\cdot\bar{\Hat{g}}$.

\begin{Prop}
Assume that $f,g\in\cL$ and both of them are $\RIO$-shift continuous measures. 
Then we have
\begin{equation}\label{Sch 2}
|\Hat{f}\cdot\bar{\Hat{g}}(E)| \le \ssqrt{|\Hat{f}|^2(E)}\ssqrt{|\Hat{g}|^2(E)}
\end{equation}
for all $E\in\cB(\RI)$.
\end{Prop}

\begin{proof}
Equation \eqref{L1 conv} ensures that
\[
\Hat{f}\cdot\bar{\Hat{g}}(p_n^{-1}(E_n))=\lim_{k\to\I}\int_{E_n}\int_{\R^k}
\Hat{f}_{n+k}(\xi,\xi')\overline{\Hat{g}_{n+k}(\xi,\xi')}d\xi' d\xi .
\]

Thus we obtain \ref{Sch 2} as in the proof of Proposition\ref{Sc eq}.
\end{proof}

Let $\cLF$ be the totality of the Fourier transform of $\RIO$-shift continuous square roots of measures.  
A sesquilinear form $\la \cdot,\cdot\ra_{\cfl^2}:\cLF\times\cLF \mapsto \C$
is defined by $< \Hat{f},\Hat{g} >_{\cfl^2}=\Hat{f}\cdot \bar{\Hat{g}}(\RI)$
and a seminorm is defined by
$\|\Hat{f}\|_{\cfl^2}=\sqrt{< \Hat{f},\Hat{f}>_{\cfl^2}}$.  
By definition, $\|\Hat{f}\|_{\cfl^2}=\|f\|_{\cl^2}$
holds.

Let $f,g\in\cL$ are $\RIO$-shift continuous measures. $\Hat{f}$ is said to be equivalent to 
$\Hat{g}$ if $\|\Hat{f}-\Hat{g}\|_{\cfl^2}=0$. It is easy to see that $f$ is equivalent to to $g$ if and only if so are $\Hat{f}$ and $\Hat{g}$.  
Let $\mathcal{F}\LL$ be the quotient set of $\cLF$ by this equivalence relation. 

A Fourier transform of $f\in \LL$ is defined by the equivalence class of 
the Fourier transform of a representative element of $f$, and denoted by $\Hat{f}$.
By definition, the Fourier transform maps $\LL$ to $\mathcal{F}\LL$.
Thus $\mathcal{F}\LL$ is a Hilbert space. From above, our Fourier transform is formulated as a unitary operator 
between two Hilbert spaces. 

Suppose that $w \in L^2(\RI,\cB(\RI) ,|\hat{f}|^2)$ and $f \in \mathcal{F}\LL$. 
We call the multiplication of $w$ and $\hat{f}$ to be a linear functional satisfying
\[
\la w\times \hat{f} , \hat{g} \ra = \int_{\RI} w(\xi) (\Hat{f}\cdot\bar{\Hat{g}} ) (d\xi)
\]
for arbitrary $g\in \LL$. 

By applying Minlos' and Sazanov' theorems, we are able to discuss the support of the Fourier transform of the square roots of measures.

\begin{Th}\label{Minlos lem}
Let $S \subset \RI$ be a nuclear space and $f \in L^2(\RI)$ be a $S$-shift continuous square root of a measure. It follows that
\[
 |\hat{f}|^2(\RI\setminus S^{'})=0.
\]
where $S^{'}$ is the dual space of $S$.
\end{Th}

\begin{Th}\label{Sazanov lem}
Suppose that $a=\{a_n\}$ and $b=\{b_n\}$ are positive sequences satisfying \\ 
$\displaystyle \sum_{n=1}^{\infty} a^2_n b^2_n <\infty $. If $f\in L^2(\RI)$ be a $H_a$-shift continuous square root of a measure, then it follows that $|\hat{f}|^2(\RI\setminus H_a)=0$.
\end{Th}

Let $f\in L^2(\RI)$ be a $H_a$-shift continuous square root of a measure and $\rho \in R_0^{\I}$. Since 
\[
\la \mathcal{F}(\tau_\rho f), \hat{g} \ra = \int_{\RI}  e^{2\pi\sqrt{-1} \la \xi, \rho \ra} (\Hat{f}\cdot\bar{\Hat{g}} )(d\xi)
\] 
for arbitrary $g \in \LL$, we have $\mathcal{F}(\tau_\rho f)= e^{2\pi\sqrt{-1} \la \xi, \rho \ra} \times \Hat{f}$.
Let $a=\{a_n\}$ and $b=\{b_n\}$ are positive sequences satisfying 
$\displaystyle \sum_{n=1}^{\infty} a^2_n b^2_n <\infty $ and $f\in L^2(\RI)$ be a $H_a$-shift continuous square root of a measure.
For arbitrary $\rho \in H_a$, let $\{\rho_k\}_{k=1}^{\I} \subset \RIO$ be a sequence such that $\displaystyle \lim_{k\to\I} \| \rho_k - \rho \|_{H_a} =0$. 
Since $|\hat{f}|^2(\RI\setminus H_a)=0$ by theorem \ref{Sazanov lem},  we have
\[
\| \tau_{\rho_m} f - \tau_{\rho_n} f \|^2 = \int_{H_b} | e^{2\pi\sqrt{-1} \la \xi, \rho_m \ra} -e^{2\pi\sqrt{-1} \la \xi, \rho_n \ra} |^2 |\hat{f}|^2(d\xi).
\]
Because $f$ is the $H_a$-shift continuous square root of a measure, $\| \tau_{\rho_m} f - \tau_{\rho_n} f \|^2 \to 0$ as $m\to \I, n\to\I$.
This shows that $e^{2\pi\sqrt{-1} \la \xi, \rho_n \ra},~ n=1,2, \cdots$ is a Cauchy sequence in the norm topology of  $L^2(H_b, \mathcal{B}(H_b),|\hat{f}|^2)$. 
We call the limit of these functions as $E(\rho)\in L^2(H_b, \mathcal{B}(H_b),|\hat{f}|^2)$. From this notation, we have
\[
\mathcal{F}(\tau_\rho f)= E(\rho)\times \Hat{f}.
\]
It is easy to check $E(\rho_1 + \rho_2)=E(\rho_1) E(\rho_2)$ for all $\rho_1,\rho_2 \in H_a$.

\begin{Prop}
Let $a=\{a_n\}$ be a positive sequence, $\mu$ be a complex Borel measure on $H_a$ and $f \in L^2(\RI)$ be $H_a$-shift continuous. We have
\begin{equation}
\int_{H_b} (\tau_\rho f )\mu(d\rho) = \mu \ast f .\label{Bochner rep}
\end{equation}
Here the left hand side is defined via the Bochner integral. 
\end{Prop}

\begin{proof}
It follows immediately that
\[
\la \int_{H_b} (\tau_\rho f) \mu(d\rho) ,g \ra =  \int_{H_a} \la \tau_\rho f, g \ra \mu (d\rho).
\]
On the other hand, assuming that $\{f_n\}$ and $\{g_n\}$ are a representative of $f$ and $g$ respectively, we have 
\[
\langle \mu \ast f,g \rangle = \lim_{n \to \I}  \int_{H_a} \la \tau_{p_n(\rho)} f_n , g_n \ra_{L^2(R^n)} \mu(d\rho) = \int_{H_a} \la \tau_{\rho} f , g \ra \mu(d\rho) .
\] 
These two equations follow for arbitrary $g \in L^2(\RI)$, \eqref{Bochner rep} is proved. 
\end{proof}

The Fourier transform of convolution product $\mu \ast f$ is equal to the multiplication of the characteristic function $\mu$ to the Fourier transform of $f$. 
\begin{Th}
Suppose that $a=\{a_n\}$ and $b=\{b_n\}$ are positive sequences satisfying \\ 
$\displaystyle \sum_{n=1}^{\infty} a^2_n b^2_n <\infty $. 
Let $\mu$ be a complex Borel measure on $H_a$ such that the characteristic function $\hat{\mu}$ is continuously extendable over $H_b$ and 
$f \in L^2(\RI)$ be $H_a$-shift continuous. Then we have
\[
\mathcal{F}(\mu \ast f) = \hat{\mu} \times \hat{f} .
\]
\end{Th}

\begin{proof}
By equation \eqref{Bochner rep}, we have
\begin{align*}
\mathcal{F}\left(\int_{H_b} (\tau_\rho f )\mu(d\rho) \right) &= \int_{H_b} \mathcal{F}(\tau_\rho f) \mu(d\rho) = \int_{H_b} E(\rho) \times \hat{f} \mu(d\rho) \\
&=\int_{H_b}  E(\rho) \mu(d\rho) \times \hat{f} = \hat{\mu} \times \hat{f} .
\end{align*}
\end{proof}

%%%%%%%%%%%%%%%%%%%%%%%%%%%%%%%%%%%%%%%%%%%%%%%%%%%%%%%%%%%%%%%%%%%%%%%%%%%%
%%%%%%%%%%%%%%%%%%%%%%%%%%%%%%%%%%%%%%%%%%%%%%%%%%%%%%%%%%%%%%%%%%%%%%%%%%%%
\section{\bf differential calculus for square roots of measure}
%%%%%%%%%%%%%%%%%%%%%%%%%%%%%%%%%%%%%%%%%%%%%%%%%%%%%%%%%%%%%%%%%%%%%%%%%%%%
%%%%%%%%%%%%%%%%%%%%%%%%%%%%%%%%%%%%%%%%%%%%%%%%%%%%%%%%%%%%%%%%%%%%%%%%%%%%
\setcounter{equation}{0}

In Averbuh-Smolyanov-Fomin \cite{Averbuh}, the differentiation (in the sense of Fomin) of complex measures  is formulated as follows.
\begin{Defi}
Let $\mu \in M(\RI)$ and $\rho \in H_a$. $\mu$ is called differentiable in the direction $\rho$ if 
\[
\lim_{t\to 0} \left \| \frac{\tau_{t\rho}\mu - \mu} {t} -\lambda \right\|= 0
\]
for some $\lambda \in M(X)$. $\lambda$ is denoted by $\partial_\rho \mu$ and called the directional derivative of $\mu$ in the direction $\rho$.
\end{Defi}

Definition of the differentiation for square roots of measures is parallel to the case of complex measures.
\begin{Defi}
Let $f\in L^2(\RI)$ be a $H_a$-shift continuous square root of a measure and $\rho \in H_a$. $f$ is called  differentiable in the direction $\rho$ if 
\[
\lim_{t\to 0} \left\| \frac{\tau_{t\rho}f - f} {t} -g \right\| = 0
\]
for some $g\in L^2(\RI)$. $g$ is denoted by $\partial_\rho f$ and called the directional derivative of $f$ in the direction $\rho$. 
If $\partial_\rho f$ exists for all $\rho \in H_a$, $f$ is called $H_a$ differentiable. 
\end{Defi}

This two differentiation is related with the following chain-rule. By using this, we are able to translate the properties of 
differentiation for the square roots of measures into those for the ordinary measures, or vice versa.
\begin{Prop}[chain-rule]
Let $f, g \in L^2(\RI)$ are $H_a$-shift continuous square roots of measures and $\rho \in H_a$. If $f,g$ are differentiable in the direction $\rho$, the product measure $f \cdot g$ are 
also differentiable in the direction $h$, and it follows that
\[
\partial_\rho (f \cdot g) = \partial_\rho f \cdot g + f \cdot \partial_\rho g.
\]  
Since $\partial_\rho (f \cdot g)(\RI)=0$, we also have 
\[
\la \partial_\rho f, g \ra = - \la f, \partial_\rho g \ra .
\]
\end{Prop}

\begin{proof}
Let $\rho \in H_a$ and $f, g \in L^2(\RI)$ are differentiable square roots of measures in the direction $\rho$. We have
\begin{align*}
& \left\|\frac{  \tau_{t\rho} (f\cdot g) - (f \cdot g) }{t}- \partial_{\rho} f\cdot g - f \cdot \partial_{\rho} g \right \| 
= \left\| \tau_{t\rho}f \cdot \frac{\tau_{t\rho}g-g}{t} -f\cdot\partial_{\rho}g \right\| + \left\|\frac{\tau_{t\rho}f -f }{t} \cdot g -\partial_{\rho} f \cdot g \right\|    \\
& \le \|\tau_{t\rho} f-f \| \left\| \frac{\tau_{t\rho}g-g}{t} \right\| + \|f\| \left\| \frac{\tau_{t\rho}g-g}{t} - \partial_{\rho} g \right\| + \|g\| \left\| \frac{\tau_{t\rho}f-f}{t} - \partial_{\rho} f \right\| \to 0 \text{~as~} t \to 0
\end{align*}
by Schwarz inequality \eqref{Schwarz}. This finishes the proof.
\end{proof}

Suppose that $a=\{a_n\}$ and $b=\{b_n\}$ are positive sequences satisfying 
$\displaystyle \sum_{n=1}^{\infty} a^2_n b^2_n <\infty $. 
Let $\rho \in H_a$ and $f \in L^2(\RI)$ be a  differentiable square root of a measure in the direction $\rho$. Since 
\[
\left\| \frac{\tau_{t\rho} f- f}{t} - \partial_{\rho} f \right\| = \left\| \frac{E(t\rho)-1}{t} \times \Hat{f} -\mathcal{F} (\partial_{\rho} f) \right\| \to 0 \text{~as~} t \to 0 ,
\]
and $|\Hat{f}|^2(R^{\I}\setminus H_b)=0$,  we have 
\[
\left\| \frac{E(s\rho)-1}{s} \times \Hat{f} -\frac{E(t\rho)-1}{t} \times \Hat{f} \right\|^2 = \int_{H_b} \left| \frac{E(s\rho)-1}{s} - \frac{E(t\rho)-1}{t}  \right|^2 d|\hat{f}|^2 \to 0 \text{~as~} s,t \to 0 .
\]
Thus there exists $W(\rho) \in L^2(H_b,\mathcal{B}(H_b),|\Hat{f}|^2)$ such that $\displaystyle \lim_{t\to 0} \left\| t^{-1}(E(t\rho)-1)- 2\pi \sqrt{-1} W(\rho) \right\| =0.$ 
This shows that 
\[
\mathcal{F}(\partial_\rho f) = 2 \pi \sqrt{-1} W(\rho) \times \Hat{f} .
\]
We see at once that 
\begin{align*}
&(1)W(\rho_1 + \rho_2)= W(\rho_1) +W(\rho_2) \text{~for all~} \rho_1,\rho_2 \in H_a,  ~~~~~~~~~~~~~~~~~~~~~~~~~~~~~~~~~~~~~~~~~~~~~~~~~~~~~~~\\
&(2)W(\alpha \rho)=\alpha W(\rho) \text{~for all~} \alpha \in C, \rho \in H_a, \\
&(3)W(\rho) = \la \xi , \rho \ra  \text{~for all~} \rho \in R_0^{\I} , \\
&(4)E(\rho)=e^{2\pi \sqrt{-1} W(\rho)} \text{~for all~} \rho \in H_a.
\end{align*}

\begin{Prop}
Let $f \in L^2(\RI)$ be a differentiable square root of a measure in the direction $\rho \in H_a$ and $A \in \cB(\RI)$. 
Then $|f|^2(A)=0$ implies that $|\partial_\rho f |^2(A)=0$. So $|\partial_\rho f |^2$ is absolutely continuous with respect to $|f|^2$.
\end{Prop}

\begin{proof}
Since $|f|^2(A)=0$ and $\left( f \cdot \tau_{t\rho} f \right) (A)=0$, we have
\[
|\partial_\rho f |^2(A) = \lim_{t \to 0} \left| \frac{\tau_{t\rho}f - f}{t} \right|^2 (A) = \left( \frac{d}{dt} \sqrt{|f|^2(A+t\rho)} \right)^2\Big|_{t=0} .
\] 
Because $|f|^2(A+t\rho)$ attains minimum value $0$ when $t=0$, derivative of $\sqrt{|f|^2(A+t\rho)}$ at $t=0$ is equal to $0$. Thus we have
$|\partial_\rho f |^2(A)=0$.
\end{proof}

\begin{Prop}
Let $f \in L^2(\RI)$ be a differentiable square root of a measure in the direction $\rho \in H_a$. Then $\partial_\rho f$ is also $H_a$-shift continuous.
\end{Prop}

\begin{proof}
For $\sigma \in \RIO$ we have
\[
\|\tau_\sigma (\partial_\rho f)  - \partial_\rho f \|^2 = \int_{\RI} | e^{2\pi \sqrt{-1} W(\sigma)} -1 |^2 |W(\rho)|^2 d|\hat{f}|^2.
\]
Since $W(\rho) \in L^2(H_b,\mathcal{B}(H_b),|\Hat{f}|^2)$ ($b=\{b_n\}$ is a positive sequence satisfying 
$\displaystyle \sum_{n=1}^{\infty} a^2_n b^2_n <\infty $) , for all $\epsilon >0$ there exists $R>0$ such that 
\[
\int_{|W(\rho)| > R} |W(\rho)|^2 d|\Hat{f}|^2 < \epsilon/2.
\]
Thus we obtain
\begin{align*}
\|\tau_\sigma (\partial_\rho f) & - \partial_\rho f \|^2 = \left(\int_{|W(\rho)| \le R} + \int_{|W(\rho)| > R} \right)| e^{2\pi \sqrt{-1} W(\sigma)} -1 |^2 |W(\rho)|^2 d|\hat{f}|^2  \\
& \le R^2 \int_{\RI}  | e^{2\pi \sqrt{-1} W(\sigma)} -1 |^2  d|\hat{f}|^2 + 2\int_{|W(\rho)| > R} |W(\rho)|^2 d|\hat{f}|^2 \le R^2 \|\tau_\sigma f -f \|^2 +\epsilon.
\end{align*}
Since $\|\tau_\sigma f -f \|^2 \to 0$ as $\|\sigma\|_{H_a} \to 0$, we have $\displaystyle \lim_{\|\sigma\|_{H_a} \to 0} \|\tau_\sigma (\partial_\rho f)  - \partial_\rho f \|^2 \le \epsilon$ for all $\epsilon > 0$.
It means that $\displaystyle \lim_{\|\sigma\|_{H_a} \to 0} \|\tau_\sigma (\partial_\rho f)  - \partial_\rho f \|^2 =0$.
\end{proof}
 
\begin{Defi}(Fr\'{e}chet derivative)
Let $f \in L^2(\RI)$ be $H_a$-shift continuous. If there is a bounded linear operator $df: H_a \mapsto L^2(\RI)$ such that
\[
\lim_{\|\rho\|_{H_a} \to 0} \|\tau_\rho f - f - df(\rho) \| / \|\rho\|_{H_a} =0 ,
\]
$f$ is said to be Fr\'{e}chet differentiable in the direction of $H_a$ and $df$ is called $H_a$ Fr\'{e}chet derivative in the direction of $H_a$.
Moreover if there is a bounded bilinear form $d^2 f : H_a \times H_a \mapsto L^2(\RI)$ such that
\[
\lim_{\|\rho_2\|_{H_a} \to 0}\|\tau_{\rho_2} (df(\rho_1)) - df(\rho_1) - d^2 f(\rho_1,\rho_2) \| / \|\rho_2\|_{H_a} =0
\]
for all $\rho_1 \in H_a$, $f$ is said to be twice Fr\'{e}chet differentiable in the direction of $H_a$ and $d^2 f$ is called second order
Fr\'{e}chet derivative in the direction of $H_a$.
\end{Defi}
 
The following theorem gives a sufficient condition for the existence of  $H_a$ Fr\'{e}chet derivatives of square roots of  measures.
\begin{Lem}\label{crossterm}
Let $f \in L^2(\RI)$ be a real-valued $R_0^{\I}$-shift continuous square root of a measure. Then
\[
\int_{\RI} \xi_i \xi_j d|\Hat{f}|^2 = 0, \text{~for all~} i,~j \in N, ~i\neq j .
\]
\end{Lem}

\begin{proof}
For simplicity we take $i=1$ and $j=2$.
Let $\{f_n\} \in \cL$ be a  representative of $f$. Since $f_n ~(n=1,2,\cdots)$ are real-valued functions, it follows that $\overline{\Hat{f_n}(\xi)} = \Hat{f_n}(-\xi)$, 
and we have
\begin{align*}
\int_{\RI} \xi_1 \xi_2 d|\Hat{f}|^2 &= \lim_{k\to\I} \int_{R^{2+k}} \xi_1 \xi_2 |\Hat{f}(\xi_1,\xi_2,\xi')|^2 d\xi_1 d\xi_2 d\xi'  \\
&= \lim_{k\to\I} \int_{R} \xi_1 \left( \int_{R^{1+k}} \xi_2 \Hat{f}(\xi_1,\xi_2,\xi') \Hat{f}(-\xi_1,-\xi_2,-\xi') d\xi_2 d\xi' \right) d\xi_1.
\end{align*}
Because the integrand is an odd function with respect to $\xi_1$ for all $k \in N$,  the right hand side is equal to $0$. 
\end{proof}

\begin{Th}
Let $f\in L^2(\RI)$ be a real-valued $H_a$ differential square root of a measure. It follows that
\[
\| \partial_\rho f \|^2 \le \left( \sup_{n\ge 1}  \int_{\RI} a^{-2}_n \xi^2_n d|\hat{f}|^2 \right) \|\rho\|^2_{H_a},
\]
for all $\rho \in H_a$ and if $\displaystyle \sup_{n\ge 1}  \int_{\RI} a^{-2}_n \xi^2_n d|\hat{f}|^2 < \I$, the mapping $\rho \mapsto \partial_{\rho} f$ 
is a one-to-one linear bounded operator from $H_a$ to $L^2(\RI)$.
\end{Th}

\begin{proof}
Let $\rho \in R_0^{\I}$. By lemma\ref{crossterm} we have
\begin{align*}
\|\partial_\rho f \|^2 &= \int_{\RI} |W(\rho)|^2 d|\hat{f}|^2 =\sum_{k=1}^{\I} \int_{\RI}  \rho^2_k \xi^2_k d|\hat{f}|^2 + 2\sum_{i>j} \int_{\RI} \rho_i \rho_j \xi_i \xi_j d|\hat{f}|^2  \\
&= \sum_{k=1}^{\I}a^2_k \rho^2_k \int_{\RI} a^{-2}_k \xi^2_k d|\hat{f}|^2 \le \left( \sup_{n\ge 1}  \int_{\RI} a^{-2}_n \xi^2_n d|\hat{f}|^2 \right) \|\rho\|^2_{H_a}.
\end{align*}
Since $\rho \mapsto \partial_{\rho} f$ is a linear mapping from $H_a$ to $L^2(\RI)$, it is continuously extendable uniquely over $H_a$.  
\end{proof}

We write $B_a$ for Banach space $\{\xi \in \RI :\displaystyle  \sup_{n\ge 1}   |a^{-1}_n \xi_n| <\I \}$, and write $\|\xi\|_{B_a}=\displaystyle \sup_{n\ge 1}   |a^{-1}_n \xi_n| $.
If $\displaystyle \sum_{n=1}^{\I} a_n b_n <\I$ for positive numbers $\{b_n\}$, we have $B_a \subset H_b$ because
\[
\|\xi\|^2_{H_b} = \sum_{n=1}^{\I} a_n^2 b^2_n a_n^{-2} \xi_n^2 \le \|\xi\|^2_{B_a} \sum_{n=1}^{\I} a_n b_n.
\]

%%%%%%%%%%%%%%%%%%%%%%%%%%%%%%%%%%%%%%%%%%%%%%%%%%%%%%%%%%%%%%%%%%%%%%%%%%%%
%%%%%%%%%%%%%%%%%%%%%%%%%%%%%%%%%%%%%%%%%%%%%%%%%%%%%%%%%%%%%%%%%%%%%%%%%%%%
\section{\bf L\'{e}vy Laplacian and Fourier transform}
%%%%%%%%%%%%%%%%%%%%%%%%%%%%%%%%%%%%%%%%%%%%%%%%%%%%%%%%%%%%%%%%%%%%%%%%%%%%
%%%%%%%%%%%%%%%%%%%%%%%%%%%%%%%%%%%%%%%%%%%%%%%%%%%%%%%%%%%%%%%%%%%%%%%%%%%%
\setcounter{equation}{0}

In this section we introduce the L\'{e}vy Laplacian for square roots of measures on a classical Wiener space $C_0[0,T]$, the space of 
all real-valued continuous functions on $[0,T]$ which start at the origin. To do this, we will give an embedding of $C_0[0,T]$ into $\RI$ in the following way.

For $s \in R$, let 
\[
h_s = \{ a=\{a_n\}_{n=1}^{\I}\in R^\I : ~~ \sum_{n=1}^{\I}(1+n^2)^s|a_n|^2 <\I \}.
\]
 $h^s$ is a Hilbert space equipped with the inner product and the norm
\[
\la a,b \ra_s = \sum_{n=1}^{\I}(1+n^2)^s a_n b_n ~~\|a\|^2_s = \sum_{n=1}^{\I}(1+n^2)^s a^2_n~~ a,b \in h_s, 
\]
and let $\displaystyle \mathbf{s} = \bigcap_{n=1}^{\I} h^n $ be a locally convex topological vector space equipped with the 
seminorms $\|\cdot\|_s ~(s\in R)$. $\mathbf{s}$ is called the space of rapidly decreasing sequence and it is known to be
 a nuclear space. Let $\mathbf{s}^{'}$ be a dual space of $\mathbf{s}$ whose duality is given by the bilinear form
\[
\la a,b \ra = \sum_{n=1}^{\I} a_n b_n ,~a\in \mathbf{s},~b \in \mathbf{s}^{'} .
\]
Let $\{e_n\}$ be a complete orthonormal system of $H_0^1[0,T] = \{ \phi : \phi, \phi^{'} \in L^2[0,T] , \phi(0)=0 \} $
\[
e_0(s) = \frac{s}{\sqrt{T}}\lower1ex\hbox{,}~~ e_n(s) = \frac{\sqrt{2T}}{n\pi} \sin \frac{n\pi s}{T}\lower1ex\hbox{,}~~~n=1,2,\cdots ,
\]
and for $a=\{a_n\} \in \mathbf{s}^{'}$, we make one-to-one correspondence
$
a \mapsto \displaystyle \phi(s)= \sum_{n=1}^{\I} a_n e_n(s) 
$
from $\mathbf{s}^{'}$ to $\mathcal{D}^{'}[0,T]$: totality of distributions on $[0,T]$. We also call the image of $\mathbf{s}$ as $D[0,T]$.

Via this correspondence, we identify the Borel subset of $\mathbf{s}^{'}$ with the Borel subset of $\mathcal{D}^{'}[0,T]$ such as
\[
\mathcal{B}(\mathbf{s}^{'}) \ni A \mapsto \{ \phi \in \mathcal{D}^{'}[0,T] : (\la \phi, e_n \ra, \la \phi,e_2 \ra, \cdots ) \in A \} \in \mathcal{B}(\mathcal{D}^{'}[0,T]) .
\]
Thus we are able to identify square roots of measures on $\mathcal{D}^{'}[0,T]$ with those on $\mathbf{s}^{'}$. In this sense we will denote by $L^2(D^{'}[0,T])$ 
the set of the square roots of measures on $D^{'}[0,T]$ as the image of $L^2(\mathbf{s})$. The notion of the differentiation, the multiplication operator , and the convolution product are defined in the same manner. As for the Fourier transform, if $f$ be a $\mathbf{s}$-shift continuous square root of measure on $\mathbf{s}^{'}$, 
then $|\Hat{f}|^2(\RI \setminus \mathbf{s}^{'})=0$ by theorem \ref{Minlos lem}. Thus if $f \in L^2(D^{'}[0,T])$ is a $D[0,T]$-shift continuous square root of a measure,
then $|\Hat{f}|^2$ is defined as a positive bounded Borel measure on $D^{'}[0,T]$. 

Note that for $\phi \in C_0[0,T]$ we have 
\[
\la \phi, e_n \ra = \sqrt{\frac{2}{T}}(-1)^k\phi(T) + \sqrt{\frac{2}{T}}\int_{0}^{\pi} \phi\left( \frac{T}{\pi}s \right)k \sin ks ds.
\]

\begin{Defi}(L\'{e}vy-Laplacian)
Let $f \in L^2(D^{'}[0,T])$ be $H_0^1[0,T]$-shift continuous and $f$ is twice differentiable in the direction $e_n$ for all $n$. 
If $\displaystyle \frac{1}{n}\sum_{k=0}^{n} \partial^2_{e_k} f$ approaches to some $g \in L^2(D^{'}[0,T])$ in the norm topology of $L^2(D^{'}[0,T])$, 
we write $g = \Delta_L f$ and the operator $\Delta_L$ is called the L\'{e}vy Laplacian for square roots of measures with respect to the CONS $\{e_n\}_{n=0}^{\I}$.
\end{Defi}

We will examine the domain of the L\'{e}vy Laplacian for square roots of measures by using Fourier transform. If $f$ belongs to the domain of 
L\'{e}vy Laplacian,  we will have 
\[
\mathcal{F}(\Delta_L f) =  -\lim_{n\to\I} 4\pi^2 \frac{\la \phi,e_0 \ra^2 + \cdots \la \phi,e_n \ra^2}{n} \times \Hat{f}
\]
when the limit of the right hand side exists. 
So we first provide a sufficient condition for existence of the limit of $\displaystyle \frac{1}{n}\sum_{k=0}^{\I} \la \phi,e_k \ra^2 $.

\begin{Defi}
Let $\Delta$ be a partition of the interval $[0,T] :0 =v_0 < v_1 < \cdots < v_{n-1} < v_n = T$,  $|\Delta|= \displaystyle \max_{1\le i \le n} |v_i-v_{i-1}|$.
We call 
\[
Q_T(\phi,\Delta) = \sum_{i=1}^{n} |\phi(v_i)-\phi(v_{i-1})|^2
\]
the quadratic variation of $\phi \in C_0[0,T]$ over $[0,T]$ with partition $\Delta$. 
If $Q_T(\phi,\Delta)$ converges to some limit as $|\Delta| \to 0$, we call this the quadratic variation of $\phi \in C_0[0,T]$ over $[0,T]$
and write $\langle \phi \rangle_T$.
\end{Defi}

\begin{Th}\label{symbol Lap}
If there exists the quadratic variation of $\phi \in C_0[0,T]$ over $[0,T]$, it follows that  
\[
\lim_{n\to\I} \frac{\la \phi,e_0 \ra^2 + \cdots \la \phi,e_n \ra^2}{n} = \frac{\langle \phi \rangle_T}{T} .
\]
\end{Th}
To prove this theorem, we make use of the theory of summability method.

\begin{Defi}
Let $\displaystyle \sum_{n=0}^{\I} u_n$ be a real-valued series such that $\displaystyle f(x) = \sum_{n=0}^{\I} u_n x^n$ converges 
for all $0 \le x < 1$. If $f(x)$ converges to a some limit $s$ as $x\to 1-0$, we say that $\displaystyle \sum_{n=0}^{\I} u_n$
Abel converges to $s$ and write $A-\displaystyle \sum_{n=0}^{\I} u_n = s$.
Let $s_n=a_1+\cdots+a_n (n=0,1,\cdots)$, $f(x)$ is also written by
\[
f(x)=(1-x)\sum_{n=0}^{\I}s_n x^n ~~~(0 \le x < 1) .
\]
If $\displaystyle \frac{1}{n} \sum_{k=0}^{n} s_n$ converges to a some limit $s$ as $n\to\I$, we say that 
$\displaystyle \sum_{n=0}^{\I} u_n$ Ces\`{a}ro converges to $s$ and write $(C,1)-\displaystyle \sum_{n=0}^{\I} u_n = s$. 
\end{Defi}

\begin{Prop}\label{HL}(Hardy-Littlewood)
If $\displaystyle A-\lim_{n\to\I}s_n = s$ and $s_n \ge 0$ for all $n$, then \\ $\displaystyle (C,1)-\lim_{n\to\I}s_n =s$.
\end{Prop}

proof of theorem \ref{symbol Lap}:
Proposition \ref{HL} shows that 
\[
\lim_{n\to\I} \frac{\la \phi,e_0 \ra^2 + \cdots \la \phi,e_n \ra^2}{n}  = \lim_{x\to 1-0}(1-x)\sum_{n=0}^{\I} \la \phi,e_n \ra^2 x^n
\]
if the limit of the right hand side exists. Remember that 
\[
\la \phi, e_n \ra = \sqrt{\frac{2}{T}}(-1)^k\phi(T) + \sqrt{\frac{2}{T}}\int_{0}^{\pi} \phi\left( \frac{T}{\pi}s \right)k \sin ks ds
\]
for all $\phi \in C_0[0,T]$. For $0\le x <1$ we have
\begin{align*}
 (1-x)\sum_{n=0}^{\I}& \la \phi,e_n \ra^2 x^n = \frac{2}{T} \bigl\{  \phi^2(T) 
 + 2(1-x)\int_{0}^{\pi}\sum_{k=0}^{\I} (-1)^k\phi(T)\phi\left( \frac{T}{\pi}s \right) kx^k \sin ks ds \\
 &+ (1-x)\sum_{k=0}^{\I}x^k \left( \int_{0}^{\pi} \phi\left( \frac{T}{\pi} s \right)k\sin ks ds \right)^2 \Bigr\}.
\end{align*}
We write $I_1(x)$ for the second term and $I_2(x)$ for the third term.
Let $P_0(x,s)$ be the Poisson kernel and $P_1(x,s)$ be as follows.
\begin{align*}
P_0(x,s) &= \frac{1}{\pi} \frac{1-x^2}{1-2x\cos s+x^2} ~~~(0\le x \le 1, 0 \le s \le \pi )  \\
P_1(x,s) &= -(1-x)\frac{\partial}{\partial s}P_0(x,s) ~~~(0\le x \le 1, 0 \le s \le \pi ).
\end{align*}

Since by
\[ 
\sum_{k=0}^{\I}k x^k \sin ks = -\frac{\partial}{\partial s} \left( \frac{1}{2} + \sum_{k=0}^{\I}x^k \cos ks \right) 
=-\frac{\pi}{2} \frac{\partial}{\partial s}P_0(x,s), 
\]
we have
\begin{align*}
I_1(x) =& -\frac{4}{T}(1-x)\phi(T) \int_{0}^{\pi}\phi\left(\frac{T}{\pi}s \right)\sum_{k=0}^{\I} k x^k \sin k(\pi-s)  ds \\
&=-\frac{2\pi}{T} 
\phi(T)\int_{0}^{T}P_1(x,\pi -s)\phi\left( 
\frac{T}{\pi}s 
\right) ds 
\end{align*}

\begin{Lem}\label{P1 est}
For all $\phi \in C[0,\pi]$, we have
\[
\lim_{x\to 1-0} \int_{0}^{\pi}P_1(x,s)\phi(s)ds = \frac{2}{\pi}\phi(0) .
\]
\end{Lem}

\begin{proof}
Since
\[
P_1(x,s)=\frac{2}{\pi}x(1-x)(1-x^2) \frac{\sin x}{(1-2x\cos s + x^2)^2}, 
\]
we have $P_1(x,s) \ge 0~~(0\le s \le \pi,~0\le x \le 1)$. Then we have  
\begin{align*}
\int_{0}^{\pi}P_1(x,s) ds &= \left[ -(1-x)P_0(x,s) \right]_{0}^{\pi} = (1-x)P_0(x,0) - (1-x)P_0(x,\pi) \\
&=\frac{1}{\pi}(1+x)-\frac{1}{\pi}\frac{(1-x)^2}{1+x} \to \frac{2}{\pi} ~~(x\to 1-0).
\end{align*}
Hence by   
\begin{equation}\label{delta est 1}
\int_{\delta}^{\pi} P_1(x,s) ds = (1-x)\{P_0(x,\delta) - P_0(x,0) \} \to 0~~(x\to 1-0)
\end{equation}
for all  $\delta > 0$, we obtain 
\begin{align*}
\bigg| \int_{0}^{\pi} P_1(x,s) & \{ \phi(s)-\phi(0) \} ds \bigg| \\
& \le  \sup_{0<s<\delta} |\phi(s)-\phi(0)|\int_{0}^{\delta}P_1(x,s)ds + 2\sup_{0<s<\pi}|\phi(s)| \int_{\delta}^{\pi}P_1(x,s)ds
\end{align*}
for all $\delta > 0$ and $\phi \in C[0,\pi]$. Since $\phi$ is uniformly continuous, by choosing $\delta$ small enough the first term converges to $0$ as $x\to 1-0$.
Because of \eqref{delta est 1}, the second term also converges to $0$ as $x\to 1-0$. 
\end{proof}

Thus we have $\displaystyle \lim_{x\to 1-0} I_1(x) = -\frac{4}{T}\phi^2(T).$ For $I_2(x)$ we have
\begin{align*}
I_2(x) &= \frac{2}{T}(1-x)\sum_{k=1}^{\I}x^k \int_{0}^{\pi}\int_{0}^{\pi}\phi\left( \frac{T}{\pi}s \right)\phi\left( \frac{T}{\pi}t \right) k^2 \sin ks\sin kt dsdt \\
&=-\frac{1}{T}(1-x)\int_{0}^{\pi}\int_{0}^{\pi}\left|\phi\left( \frac{T}{\pi}s \right) - \phi\left( \frac{T}{\pi}t \right)\right|^2 
\sum_{k=0}^{\I}k^2 x^k \sin ks \sin kt dsdt  \\
&~~~~~~~~~+ \frac{1}{T}(1-x)\int_{0}^{\pi}\int_{0}^{\pi}\left\{ \phi^2\left( \frac{T}{\pi}s \right) + \phi^2\left( \frac{T}{\pi}t \right)  \right\}
 \sum_{k=0}^{\I}k^2 x^k \sin ks \sin kt dsdt .
\end{align*}
We write $J_1(x)$ for the first term and $J_2(x)$ for the second term.
Hence by Lemma \ref{P1 est} and
\begin{align*}
J_2(x) &= \frac{2}{T}\int_{0}^{\pi} \phi^2 \left( \frac{T}{\pi}s \right) \sum_{k=0}^{\I} \{ (-1)^k -1 \} kx^k\sin ks ds \\
&= \frac{\pi}{T}\int_{0}^{\pi} \{ P_1(x,\pi-s) + P_1(x,s) \} \phi^2 \left( \frac{T}{\pi}s \right) ds ,
\end{align*}
we have $\displaystyle \lim_{x\to 1-0} J_2(x) = \frac{2}{T}\phi^2(T).$ Taken together, if $\displaystyle \lim_{x\to 1-0} J_1(x)$ exists, we will have
\[
\lim_{x\to 1-0}(1-x)\sum_{n=0}^{\I} \la \phi,e_n \ra^2 x^n = \lim_{x\to 1-0} J_1(x) .
\]

So we will concerned with the limit of $J_1(x)$ as $x \to 1-0$. Here since 
\begin{align*}
(1-x)\sum_{k=0}^{\I}k^2 x^k \sin ks \sin kt &= \frac{1}{2}(1-x)\sum_{k=0}^{\I}k^2 x^k \{ \cos k(s-t) - \cos k(s+t) \} \\
&= \frac{1}{2}\frac{\partial}{\partial s}(1-x)\sum_{k=0}^{\I}k x^k \{ \sin k(s-t) - \sin k(s+t) \} \\
&= \frac{\pi}{4}\frac{\partial}{\partial s} \left\{ P_1(x,s-t) - P_1(x,s+t) \right\} ,
\end{align*}
we have 
\begin{align*}
J_1(x) &= \frac{\pi}{4T}\int_{0}^{\pi}\int_{0}^{\pi}\left|\phi\left( \frac{T}{\pi}s \right) - \phi\left( \frac{T}{\pi}t \right)\right|^2 P_2(x,s+t)dsdt \\
&~~~~~~~~~~-\frac{\pi}{4T}\int_{0}^{\pi}\int_{0}^{\pi}\left|\phi\left( \frac{T}{\pi}s \right) - \phi\left( \frac{T}{\pi}t \right)\right|^2 P_2(x,s-t)dsdt .
\end{align*}
Here we write $P_2(x,s)=\partial_s P_1(x,s)$.
We write $K_1(x)$ for the first term and $K_2(x)$ for the second term.

\begin{Prop}\label{est rest}
The following \eqref{a} - \eqref{c} hold.
\begin{align}
& \lim_{x\to 1-0}\int_{0}^{\pi} v|P_2(x,v)|dv < \I. ~~~~~~~~~~~~~~~~~~~~~~~~~~~~~~~~~~~~~~~~~~~~~~~~~~~~~~~~~~~~\label{a}\\
& \lim_{x\to 1-0}\int_{0}^{\pi} v P_2(x,v)dv = -\frac{2}{\pi} \label{b} .\\
& \lim_{x\to 1-0} \int_{\delta}^{\pi} v | P_2(x,v) | dv = 0 ~~\text{for all}~\delta > 0 \label{c}.
\end{align}
\end{Prop}

As in the proof of Lemma \ref{P1 est} , for all $\phi \in C[0,\pi]$ it follows that 
\[
\lim_{x\to 1-0}\int_{0}^{\pi} v P_2(x,v) \phi(v) dv = -\frac{2}{\pi} \phi(0).
\]

\begin{proof}
Proof of \eqref{a}:~~Let $0 \le \theta_x \le \pi$ be such as $P_2(x,\theta_x)=0$. Since 
\[
P_2(s,x)= \frac{2}{\pi}x(1-x)(1-x^2)\left\{ \frac{\cos s}{(1-2x\cos s+x^2)^2} - \frac{4x\sin^2 s}{(1-2x\cos s+x^2)^3} \right\},
\]
we have
\[
\cos \theta_x = \frac{\sqrt{(1+x^2)^2+32x^2}-(1+x^2)}{4x}.
\]
For fixed $0\le x <1$, hence by $P_2(x,v) \ge 0$ if and only if $0\le v \le \theta_x$, we have 
\[
\int_{0}^{\pi} v|P_2(x,v)|dv = \int_{0}^{\theta_x} vP_2(x,v)dv - \int_{\theta_x}^{\pi} v P_2(x,v)dv.
\]
Here the primitive integral of $vP_2(x,v)$ is calculated by
\begin{align*}
\int v P_2(x,v) dv &= \int v P_1^{\prime}(x,v) dv = v P_1(x,v) - \int P_1(x,v) dv \\
&= v P_1(x,v) + (1-x)\int P^{\prime}_0(x,v) dv = v P_1(x,v) + (1-x) P_0(x,v).
\end{align*}
So we have
\begin{align*}
\int_{0}^{\pi}\int_{0}^{\pi}v|P_2(x,v)|dv &= 2\theta_x P_1(x,\theta_x)+2(1-x) P_0(x,\theta_x) \\
& - (1-x)P_0(x,0) - \pi P_1(x,\pi) -(1-x) P_0(x,\pi)
\end{align*}

For the first two terms, by elementary computation it it shown that
\begin{align*}
\theta_x P_1(x,\theta_x) &= \frac{1}{4\pi} \frac{x}{1+x} \frac{\theta_x}{\sin \theta_x} \frac{6(1+x^2)\sqrt{(1+x^2)^2+32x^2} +10(1+x^2)^2+32x^2}
{(1+x^2)\sqrt{(1+x^2)^2+32x^2} +(1+x^2)^2+8x^2}\\
(1-x) P_0(x,\theta_x) &= \frac{1}{\pi} \frac{3(1+x^2)+\sqrt{(1+x^2)^2+32x^2}}{4(1+x)} \lower1.5ex\hbox{.}
\end{align*}
Since $\theta_x \to +0$ as $x\to 1-0$, we have $\displaystyle 2\theta_x P_1(x,\theta_x) \to \frac{3}{2\pi}$ and 
$\displaystyle 2(1-x) P_0(x,\theta_x) \to \frac{3}{\pi}$ as $x\to 1-0$.   
$\displaystyle (1-x)P_0(x,0) = \frac{1+x}{\pi} \to \frac{2}{\pi}$ and the last two terms are easily shown to converge to $0$.
Combining them all, we have 
\[
\lim_{x\to 1-0}\int_{0}^{\pi} v|P_2(x,v)|dv = \frac{3}{2\pi}+\frac{3}{\pi}-\frac{2}{\pi} = \frac{5}{2\pi} < \I .
\]
Proof of \eqref{b}:~Because of
\begin{align*}
\int_{0}^{\pi} v P_2(x,v)dv = \pi P_1(x,\pi) + (1-x) P_0(x,\pi) -(1-x) P_0(x,0) ,
\end{align*}
we have
\[
\lim_{x\to 1-0} \int_{0}^{\pi} v P_2(x,v)dv = -\lim_{x\to 1-0} (1-x) P_0(x,0) = -\frac{2}{\pi}.
\]

Proof of \eqref{c}:~
Since $\theta_x \to 0$ as $x\to 1-0$, we are able to take $0\le x < 1$ to be $\theta_x < \delta$ for all  $\delta > 0$.
Then  because $P_2(x,v) \le 0$ when $\theta_x \le v \le \pi$, we obtain
\begin{align*}
\int_{\delta}^{\pi} v | P_2(x,v) | dv &= - \int_{\delta}^{\pi}v P_2(x,v)dv  \\
&=\pi P_1(x,\pi) + (1-x)P_0(x,\pi) -\delta P_1(x,\delta) - (1-x)P_0(x,\delta).
\end{align*}
It is easy to check that the last four terms converge to $0$ as $x\to 1-0$, which proves the lemma.
\end{proof}

\begin{Cor}\label{K1}
\[
\displaystyle \lim_{x \to 1-0} K_1(x) = 0.
\]
\end{Cor}

\begin{proof}
Let $A_{\delta}^1 = \{0 \le s \le \pi,~0\le t \le \pi,~0 \le s+t < \delta \}$, $A_{\delta}^2 = \{0 \le s \le \pi,~0\le t \le \pi,~2\pi -\delta < s+t \le 2\pi \}$
, and $D_{\delta} = \{0 \le s \le \pi,~0\le t \le \pi,~\delta \le s+t \le 2\pi-\delta \}$.  The integral domain of $K_1(x)$ is divided into 
$A_\delta^1$, $A_\delta^2$, and $D_\delta$. We call the integral over $A_\delta^1$, $A_\delta^2$, and $D_\delta$ as $L_1(x), L_2(x)$ and $L_3(x)$ respectively.
We first compute $L_1(x)$.
\[
|L_1(x)| \le \sup_{(s,t)\in A_\delta^1} \left|\phi\left( \frac{T}{\pi}s \right) - \phi\left( \frac{T}{\pi}t \right)\right|^2 \iint_{A_\delta^1} |P_2(x,s+t)| dsdt .
\]
Since $(s,t) \in A_\delta^1$ implies $0 \le s \le  \delta ,~ 0 \le t \le  \delta$ and 
\[
\left|\phi\left( \frac{T}{\pi}s \right) - \phi\left( \frac{T}{\pi}t \right)\right| \le \left|\phi\left( \frac{T}{\pi}s \right) - \phi(0) \right| +
\left|\phi (0) - \phi\left( \frac{T}{\pi}t \right)\right| ,
\]
we have 
\[
\sup_{(s,t)\in A_\delta^1} \left|\phi\left( \frac{T}{\pi}s \right) - \phi\left( \frac{T}{\pi}t \right)\right|^2 \le
4 \sup_{0\le u \le \pi T^{-1}\delta } \left|\phi(u) - \phi(0) \right| ^2 .
\]
On the other hand,  we have
\begin{align*}
\iint_{A_\delta^1} & |P_2(x,s+t)| dsdt \le \int_0^\pi \int_0^\pi  |P_2(x,s+t)| dsdt \\
&=\int_0^{\pi} \int_0^{v} |P_2(x,v)|dsdv + \int_{\pi}^{2\pi} \int_{v-\pi}^{\pi}|P_2(x,v)|dsdv ~~(\text{Set} ~v=s+t) \\
&=\int_{0}^{\pi}v|P_2(x,v)| dv + \int_{\pi}^{2\pi}(2\pi -v)|P_2(x,v)| dv = 2\int_{0}^{\pi}v|P_2(x,v)| dv.
\end{align*}
The last equation follows from $P_2(x,v)=P_2(x,2\pi-v)$. Because of \eqref{a}, the last term converges to $5/\pi$ as $x \to 1-0$.
Thus we have 
\[
| L_1(x) | \le \frac{20}{\pi} \sup_{0\le u \le \pi T^{-1}\delta } \left|\phi(u) - \phi(0) \right| ^2  .
\]
In the same manner we can show that 
\[
| L_2(x) | \le \frac{20}{\pi} \sup_{0\le u \le \pi T^{-1}\delta } \left|\phi(T-u) - \phi(T) \right| ^2  .
\]
We proceed to compute $L_3(x)$. We have
\begin{align*}
L_3(x) &\le 4 \sup_{0\le u \le T} |\phi(u)|^2 \iint_{D_\delta}  |P_2(x,s+t)| dsdt  \\
& =4 \sup_{0\le u \le T} |\phi(u)|^2 \left( \int_{\delta}^{\pi} \int_{0}^{v} |P_2(x,v)|dsdv + \int_{\pi}^{2\pi -\delta} \int_{v-\pi}^{\pi}|P_2(x,v)|dsdv \right) \\
&= 8 \sup_{0\le u \le T} |\phi(u)|^2 \int_{\delta}^{\pi}v|P_2(x,v)| dv.
\end{align*}
Because of \eqref{c} , the last term of the equation converges to $0$ as $x \to 1-0$ for all $\delta > 0$. Taken together,  we have
\[
\lim_{x \to 1-0}|K_1(x)| \le \frac{20}{\pi} \left( \sup_{0\le u \le \pi T^{-1}\delta } \left|\phi(u) - \phi(0) \right| ^2 +\sup_{0\le u \le \pi T^{-1}\delta } \left|\phi(T-u) - \phi(T) \right| ^2 \right) .
\]
Since $\phi$ is uniformly continuous on $[0,T]$, the left hand side goes to $0$ as $\delta \to +0$. Thus we obtain $\displaystyle \lim_{x \to 1-0} K_1(x) = 0$ .
\end{proof}

We now turn to deal with $K_2(x)$.  By letting $v=s-t$, the domain of integration $[0,\pi] \times [0,\pi]$ is transformed to 
$\{ v \le t \le \pi -v,~0 \le v \le \pi \} \cup \{ -v \le t \le \pi +v,-\pi \le v \le 0 \}$. Since the integrand is symmetric with respect to $v=0$, we have  
\begin{equation}
\lim_{x \to 1-0} K_1(x) \le \frac{\pi}{2T}\int_{0}^{\pi}\int_{v}^{\pi-v}\left|\phi\left( \frac{T}{\pi}(t+v) \right) - \phi\left( \frac{T}{\pi}t \right)\right|^2 P_2(x,v) dtdv .
\end{equation}
 
Set
\[
I(v) = -\frac{\pi}{2T} \int_{0}^{\pi-v} \frac{\left|\phi\left( \frac{T}{\pi}(t+v) \right) - \phi\left( \frac{T}{\pi}t \right)\right|^2}{|v|} dt.
\]
Then $K_1(x)$ is written by  
\[
K_2(x) = \int_{0}^{\pi} v P_2(x,v) I(v) dv .
\]
Proposition \ref{est rest} shows that $\displaystyle \lim_{x \to 1-0} K_2(x) = -\frac{2}{\pi} \lim_{v \to 0} I(v)$ if the limit of the right hand side exists.

Let $N$ be a integer. By substituting $\pi/2^N$ for $v$, we have 
\begin{align*}
I\left(\frac{\pi}{2^N}\right) &=  -\frac{\pi}{2T} \int_{0}^{\pi-\frac{\pi}{2^N}} \left| \phi\left(\frac{T}{\pi}t+\frac{T}{2^N}\right) 
-\phi\left(\frac{T}{\pi}t\right) \right|^2 \frac{2^N}{\pi} dt \\
&=-\frac{\pi^2}{2T^2} \int_{0}^{T-\frac{T}{2^N}} \left| \phi\left(t+\frac{T}{2^N}\right) - \phi(t) \right|^2 \frac{2^N}{\pi} dt.
\end{align*}
We will consider the limit of $I(\pi/2^N)$ as $N\to \I$. 

Let $M(>N)$ be an integer. We write $I_{M,N}$ for a Riemann sum approximations to the integral $\displaystyle I\left(\frac{\pi}{2^N}\right)$ such that
the interval $[0,T-\frac{T}{2^N}]$ is divided into $2^{M-N}$ subintervals, all with the same length $\frac{T}{2^M}$, i.e. 
\[
I_{M,N} = -\frac{\pi}{2T} \sum_{k=0}^{(2^N-1)2^{M-N}}\left|\phi \left( \frac{kT}{2^M} + \frac{T}{2^N} \right) -
\phi \left( \frac{kT}{2^M} \right) \right|^2 \frac{2^N}{2^M} .
\]
By letting $k= (p-1)2^{M-N}+r,~p=1,2, \cdots ,2^N-1,~r=1,2, \cdots,2^{M-N}$, we change the order of sum of $I_{M,N}$.
\[
I_{M,N} = -\frac{\pi}{2T} \left\{ \sum_{r=1}^{2^{M-N}}\sum_{p=1}^{2^N-1} \left|\phi \left( \frac{pT}{2^N} + \frac{rT}{2^M} \right) -
\phi \left( \frac{(p-1)T}{2^N} + \frac{rT}{2^M} \right) \right|^2 +  \left|\phi (0) -\phi \left( \frac{T}{2^M} \right) \right|^2  \right\}\frac{2^N}{2^M}
\]
Here we focus on the term
\[
J_r=\sum_{p=1}^{2^N-1} \left|\phi \left( \frac{pT}{2^N} + \frac{rT}{2^M} \right)-\phi \left( \frac{(p-1)T}{2^N} + \frac{rT}{2^M} \right) \right|^2 .
\]

By definition we have 
\begin{align*}
J_r &\le \sup_{|\Delta|\le \frac{T}{2^N} } Q_T(\phi,\Delta) \\
J_r &\ge \inf_{|\Delta|\le \frac{T}{2^N} } Q_T(\phi,\Delta)  
-\left\{ \left|\phi (0) -\phi \left( \frac{rT}{2^M} \right) \right|^2 + \left|\phi (T) 
-\phi \left(T-\frac{(2^{M-N}-r)T}{2^M} \right) \right|^2 \right\} \\
& \ge \inf_{|\Delta|\le \frac{T}{2^N} } Q_T(\phi,\Delta) - \sup_{0\le v\le \frac{T}{2^N}}
\left\{ |\phi(0)-\phi(v)|^2-|\phi(T)-\phi(T-v)|^2 \right\}.
\end{align*}
Thus we have 
\begin{align}
\inf_{|\Delta|\le \frac{T}{2^N} } & Q_T(\phi,\Delta)
-\sup_{0\le v\le \frac{T}{2^N}}\left\{ |\phi(0)-\phi(v)|^2-|\phi(T)-\phi(T-v)|^2 \right\} \notag \\
& \le -\frac{2T}{\pi}I\left( \frac{\pi}{2^N}\right) \le \sup_{|\Delta|\le \frac{T}{2^N} } Q_T(\phi,\Delta) \label{QV2}
\end{align}
as $M\to\I$.  By letting $N\to \I$, we can conclude that
$
\displaystyle\lim_{v\to +0} I(v) = - \frac{\pi \langle \phi \rangle_T}{2T}
$
if the quadratic variation of $\phi$ over $[0,T]$ exists. 
Thus we obtain 
\[
\displaystyle\lim_{x\to 1-0} K_2(x) =  \frac{ \langle \phi \rangle_T}{T},
\]
which proves the theorem \ref{symbol Lap}. 
\begin{flushright}
$\Box$
\end{flushright}

To apply theorem \ref{symbol Lap} to the L\'{e}vy Laplacian for square roots of measures, we need to check $\displaystyle \frac{1}{n} \sum_{k=0}^n \la \phi,e_k \ra^2$ 
converges in the norm topology of $L^2(D^{'}[0,T])$.
\begin{Lem}\label{L4}
Let $f$ be second order Fr\'{e}chet differentiable in the direction of $H_a$, and \\ $|\hat{f}|^2(D{'}[0,T] \setminus C_0[0,T])=0$.
Then 
\[
\int_{C_0[0,T]} |\phi(s) |^4 |\Hat{f}|^2(d\phi) < \I 
\]
for $0 \le s \le T$.
\end{Lem}

\begin{proof}
Let  $\eta^{\epsilon}_s \in H_0^1[0,T] (0 \le s \le T,~\epsilon > 0)$ to be the function
\[
\eta^{\epsilon}_s(u) = \begin{cases} 
                  u & (0\le u \le s-\epsilon) \\
                  -\frac{1}{2\epsilon}(u-s)^2+s-\frac{\epsilon}{2} &  (s-\epsilon \le u \le s) \\
                  s-\frac{\epsilon}{2} & (s\le u \le T).
                 \end{cases} 
\]
Since $\eta^{\epsilon}_s$ is second order differentiable, we have
\[
\langle \phi, \eta^{\epsilon}_s \rangle =  \left[ \phi(u) \dot{\eta}^{\epsilon}_s(u) \right]_0^T - \int_{0}^T \phi(u) \ddot{\eta}^{\epsilon}_s(u) du 
= \frac{1}{\epsilon} \int_{s-\epsilon}^{\epsilon} \phi(u)du 
\]
for $\phi \in C_0[0,T]$. Thus we obtain
\[
\| \partial^2_{\eta^{\epsilon}_s} f \|^2 = (2\pi)^4 \int_{C_0[0,T]} \left( \frac{1}{\epsilon} \int_{s-\epsilon}^{\epsilon} \phi(u)du \right)^4 |\Hat{f}|^2(d\phi).
\]
Because $f$ is second order Fr\'{e}chet differentiable in the direction of $H_a$, the left hand side converges to $d^2 f (\eta_s,\eta_s)$ as $\epsilon \to 0$.
Here
\[
\eta_s(u) = \begin{cases} 
                  u & (0\le u \le s) \\
                  s &  (s \le u \le T) .
                 \end{cases} 
\]
Hence by 
\[
\lim_{\epsilon \to +0}\frac{1}{\epsilon} \int_{s-\epsilon}^{\epsilon} \phi(u)du = \phi(s) 
\]
for $\phi \in C_0[0,T]$, we have
\[
\lim_{\epsilon \to 0} (2\pi)^4 \int_{C_0[0,T]} \left( \frac{1}{\epsilon} \int_{s-\epsilon}^{\epsilon} \phi(u)du \right)^4 |\Hat{f}|^2(d\phi) \ge (2\pi)^4 \int_{C_0[0,T]} |\phi(s) |^4 |\Hat{f}|^2(d\phi) ,
\]
which proves the lemma.
\end{proof}

\begin{Th}\label{LL}
Let $f\in L^2(D^{'}[0,T])$ be twice differentiable in any direction of $H_0^1[0,T]$, let
$A$ be the totality of $\phi \in C_0[0,T]$ such that the quadratic variation of $\phi$ exists
and let $D(\Delta_L)$ be 
\[
D(\Delta_L) = \{ f\in L^2(D^{'}[0,T]) : ~|\Hat{f}|^2(D^{'}[0,T] \setminus A)=0,~\int_{A} \langle \phi \rangle^2 |\Hat{f}|^2(d\phi) < \infty \} .
\]
If $f \in D(\Delta_L)$, then $\Delta_L f$ exists and it follows that
\[
\mathcal{F} (\Delta_L f) = - (2\pi)^2 \frac{\langle \cdot \rangle_T}{T} \times \hat{f} .
\]
\end{Th}

\begin{proof}
It is sufficient to show that 
\begin{equation}\label{eqr}
\lim_{x \to 1-0} \int_{A} \left| K_1(x) + K_2(x) - \frac{\langle \phi \rangle_T}{T} \right|^2 |\Hat{f}|^2(d\phi)=0
\end{equation}
for $f \in D(\Delta_L)$. 
We first show that $\displaystyle \lim_{x \to 1-0} \int_{A} |K_1(x)|^2 |\Hat{f}|^2(d\phi) =0$. Hence by the proof of Lemma \ref{K1}, we have
\begin{align*}
\lim_{x\to 1-0} \int_A |K_1(x)|^2 |\Hat{f}|^2(d\phi) & \le C  \int_A \sup_{0\le u \le \pi T^{-1}\delta } \left( |\phi(u) - \phi(0) | ^4 +|\phi(T-u) - \phi(T) | ^4 \right) |\hat{f}|^2(d\phi) \\
 & + C \lim_{x \to 1-0} \int_A \sup_{0\le u \le T} |\phi(u)|^4 |\hat{f}|^2(d\phi) \left( \int_{\delta}^{\pi} v |P_2(x,v)|dv \right)^2
\end{align*}
for some $C>0$ and for all $\delta > 0$. Lemma \ref{L4} shows that the above integrals are finite. 
Since $\sup_{0\le u \le \pi T^{-1}\delta } \left( |\phi(u) - \phi(0) | ^4 +|\phi(T-u) - \phi(T) |  \right)$ is monotonically decreasing to $0$ as $\delta \to +0$, the first term goes to $0$. 
Because of \eqref{c}, the second term converges to $0$ as $x \to +0$.

We next show that $\displaystyle\lim_{x\to 1-0} \int_A \left|K_2(x) -\frac{\langle \phi \rangle_T}{T} \right|^2 |\Hat{f}|^2(d\phi) =0$. Hence by \eqref{QV2}, 
it follows that
\begin{align*}
\lim_{x\to1-0} \left|K_2(x) -\frac{\langle \phi \rangle_T}{T} \right|  & \le \lim_{v \to +0} \left| \sup_{|\Delta|\le v } Q_T(\phi,\Delta)- \frac{\langle \phi \rangle_T}{T} \right| 
+\lim_{v \to +0} \left| \inf_{|\Delta|\le v } Q_T(\phi,\Delta)- \frac{\langle \phi \rangle_T}{T} \right|  \\ 
&+  \lim_{v \to +0} \sup_{0\le u \le v } \left( |\phi(u) - \phi(0) |  +|\phi(T-u) - \phi(T) |  \right) \\
&= \lim_{v \to +0}\sup_{|\Delta|\le v } Q_T(\phi,\Delta) - \lim_{v \to +0}\inf_{|\Delta|\le v } Q_T(\phi,\Delta)  \\
&+ \lim_{v \to +0} \sup_{0\le u \le v } \left( |\phi(u) - \phi(0) |  +|\phi(T-u) - \phi(T) |  \right).
\end{align*}
Since the last term is monotonically decreasing to $0$ as $\delta \to +0$, monotone convergence theorem shows the desired conclusion.  Taken together, 
\eqref{eqr} is proven.
\end{proof}

Our next objective is to characterize the L\'{e}vy Laplacian for square roots of measures via asymptotic spherical mean.
Let  $S_n$ be the $n$-dimensional unit sphere and $\mu_n$ be the normalized uniform measure on $S_n$. 

\begin{Defi}
Let $f\in L^2(C_0[0,T])$ be $H_0^1[0,T]$-shift continuous. Via Bochner integral of $\tau_{\rho h^{(n)}} f$ over $S_n$, we define 
\[
M^n_{\rho}f =  \int_{S_{n-1}} \tau_{\rho h^{(n)}} f  d \mu_{n-1} .
\]
Here we write $\displaystyle h^{(n)}  = \sum_{k=1}^{n}  h_k e_k$. 
If $M^n_\rho f$ converges some square root of a measure as $n\to\I$ in the norm topology of $L^2(D^{'}[0,T])$, this is called the  
asymptotic spherical mean of $f$ over the sphere of radius $\rho$
and it is written by $M_\rho f$.
\end{Defi}

\begin{Prop}\label{asm}
Let $f\in L^2(C_0[0,T])$ be $H_0^1[0,T]$-shift continuous and $|\Hat{f}|^2(D^{'}[0,T] \setminus A)=0$. Then
 $M^n_\rho f$ converges to the spherical mean $M_\rho f$ in the norm topology of $L^2(D^{'}[0,T])$ as $n \to \I$ and it follows that
\[
\mathcal{F} (M_\rho f) = \left( e^{-2\pi^2 \rho^2 T^{-1} \la \cdot \ra_T} \hat{f} \right) .
\] 
 \end{Prop}

\begin{proof}
Let $f\in L^2(C_0[0,T])$ be $H_0^1[0,T]$-shift continuous and $|\Hat{f}|^2(D^{'}[0,T] \setminus A)=0$. Then we have
\begin{align*}
\mathcal{F}(M^n_\rho f) &= \mathcal{F} \left( \int_{S_{n-1}} \tau_{ \rho h^{(n)} } f  d \mu_{n-1} \right) =  \int_{S_{n-1}} \mathcal{F} (\tau_{\rho h^{(n)}} f )  d \mu_{n-1} \\
&= \int_A \int_{S^{n-1}} e^{2\pi \sqrt{-1} \rho \la h^{(n)} , \phi \ra } \mu_{n-1}(dh^{(n)}) |\Hat{f}|^2(d\phi)
\end{align*}

Let $w_n$  be the volume of the surface area of $S_n$ and $\displaystyle r^2_n = \sum_{k=1}^{n}  \la \phi , e_k \ra^2$. Choosing some adequate unitary transform on $S_n$, we have
\begin{align*}
\int_{S_{n-1}} & e^{2\pi \sqrt{-1} \rho\la h^{(n)}, ~\phi \ra } \mu(d h^{(n)}) = \frac{1}{w_{n-1}}\int_{S_{n-1}} e^{2\pi \sqrt{-1}\rho r_n h_n } d h_1\cdots d h_n  \\
&= \frac{w_{n-2}}{w_{n-1}} \int_{-1}^{1} (1-x^2)^{\frac{n}{2} -1} e^{2\pi\sqrt{-1}\rho r_n x } dx  = 2 \frac{w_{n-2}}{w_{n-1}} \int_{0}^{1} (1-x^2)^{\frac{n}{2} -1} \cos (2\pi \rho r_n x)  dx .
\end{align*}

Let $I_n(\phi)$ be the right hand side of the above equation. Then we have
\begin{align*}
I_n(\phi) &= 2\frac{w_{n-2}}{w_{n-1}}  \sum_{k=0}^{\I} \frac{(-4\pi^2  \rho^2 r_n^2 )^{k}}{(2k)!} \int_{0}^{1} x^{2k} (1-x^2)^{\frac{n}{2} -1}  dx  \\
&= \frac{w_{n-2}}{w_{n-1}}  \sum_{k=0}^{\I} \frac{(-4\pi^2 \rho^2 r_n^2 )^{k}}{(2k)!} \int_{0}^{1} t^{k-\frac{1}{2}} (1-t)^{\frac{n}{2} -1}  dt ~~(\text{Set}~t=x^2)  \\
&=\frac{1}{\sqrt{\pi}} \sum_{k=0}^{\I} \frac{(-4\pi^2  \rho^2 r_n^2)^{k}}{(2k)!}\frac{\Gamma(k+ \frac{1}{2} )\Gamma^2\left(\frac{n}{2}\right)}  {\Gamma\left(k + \frac{n+1}{2} \right) \Gamma\left(\frac{n-1}{2}\right) }   \\
&= \sum_{k=0}^{\I} \frac{(-\pi^2 \rho^2 r_n^2)^{k}}{k!}\frac{\Gamma^2\left(\frac{n}{2}\right)}  { \left(k+\frac{n-1}{2} \right) \left(k+\frac{n-3}{2} \right) \cdots \frac{n+1}{2}  \Gamma\left(\frac{n+1}{2}\right)\Gamma\left(\frac{n-1}{2}\right) }  \\
&=\frac{ \Gamma^2 \left( \frac{n}{2} \right)}{ \Gamma \left(\frac{n+1}{2} \right) \Gamma\left(\frac{n-1}{2}\right)} \sum_{k=0}^{\I} \frac{b_k}{k!} \left(-\frac{2\pi^2 \rho^2 r_n^2}{n} \right)^k  .
\end{align*}
Here we write
\[
b_0=1,~~b_k= \prod_{j=1}^k \left( 1+\frac{2j-1}{n} \right)^{-1} (k \ge 1) .
\]
Since by Wallis' formula, it follows that
\[
\lim_{n\to\I} \frac{ \Gamma^2 \left( \frac{n}{2} \right)}{ \Gamma \left(\frac{n+1}{2} \right) \Gamma\left(\frac{n-1}{2}\right)} = 1 .
\]
Because of 
\[
b_k-b_{k+1} = b_k\frac{2k-1}{n} \left( 1+ \frac{2k-1}{n} \right)^{-1} \le \frac{2k-1}{n} ,
\]
we have $1-b_k \le k^2/n~(k \ge 0)$. So letting $\displaystyle x_n = -\frac{2\pi^2 \rho^2 r_n^2}{n}$ we have 
\[
\biggl| \sum_{k=0}^{\I} \frac{1-b_k}{k!} x_n^k \biggr|  \le \frac{1}{n} \sum_{k=0}^{\I} \frac{k}{(k-1)!} |x_n|^k = \frac{1}{n} (|x_n| +1) e^{|x_n|} . \\
\]
By assumption the right hand side converges to $0$ for $\phi \in A$.
Thus it follows that 
\begin{equation}\label{ae conv1}
\lim_{n\to\I} \sum_{k=0}^{\I} \frac{b_k}{k!} x_n^k = \lim_{n\to\I} \sum_{k=0}^{\I} \frac{1}{k!} x_n^k = e^{-2\pi^2 \rho^2 T^{-1} \langle \phi \rangle_T}
\end{equation}
for $\phi \in A$. In addition, because of
\[
|I_n(\phi)| = \left| 2\frac{w_{n-2}}{w_{n-1}}  \int_{0}^{1} (1-x^2)^{\frac{n}{2} -1}   \cos (2\pi  \rho r_n x) dx \right| \le 2\frac{w_{n-2}}{w_{n-1}} \int_{0}^{1} (1-x^2)^{\frac{n}{2}-1} dx =1 ,
\]
$I_n(\phi)$ is uniformly bounded. We are now in a position to show $\{ M^n_\rho f \}$ is a Cauchy sequence in the norm topology of $L^2(D^{'}[0,T])$. 
Let $m, ~n ~(m>n)$ be integers.  Because $I_n(\phi)$ converges to $e^{-2\pi^2 \rho^2 T^{-1} \langle \phi \rangle_T}$ for all $\phi \in A$ as $n\to \I$
and $I_n(\phi) \le 1$, we have
\[
\lim_{m,n \to \I} \|M^n_\rho f - M^m_\rho f \|^2 = \lim_{m,n \to \I} \int_{A} | I_m(\phi) - I_n(\phi) |^2 |\hat{f}|^2(d\phi) =0 ,
\]
which proves the proposition.
\end{proof}

The relation between the spherical mean and the L\'{e}vy Laplacian  is stated as follows.

\begin{Cor}
Assume that $f \in D(\Delta_L)$. Then we have
\[
\Delta_L f = 2 \lim_{\rho \to +0} \frac{M_\rho f -f}{\rho^2} .
\]
\end{Cor}

\begin{proof}
Let $f \in D(\Delta_L)$. Theorem \ref{LL} and Proposition \ref{asm} shows that
\begin{align*}
\left\| \Delta_L f - 2  \frac{M_\rho f -f}{\rho^2} \right\|^2 = & (2\pi)^2 \int_A \left| \frac{\la \phi \ra_T}{T} +
\frac{e^{-2\pi^2 \rho^2 T^{-1} \la \phi \ra_T} -1}{2\pi^2 \rho^2} \right|^2 |\Hat{f}|^2(d\phi).
\end{align*}
Letting $\rho \to 0$, we have the desired conclusion.
\end{proof}

%%%%%%%%%%%%%%%%%%%%%%%%%%%%%%%%%%%%%%%%%%%%%%%%%%%%%%%%%%%%%%%%%%%%%%%%%%%%
%%%%%%%%%%%%%%%%%%%%%%%%%%%%%%%%%%%%%%%%%%%%%%%%%%%%%%%%%%%%%%%%%%%%%%%%%%%%

\end{document}